\documentclass[12pt,a4paper]{amsart}
 \usepackage{amsfonts,amssymb,enumerate,color,bm,hhline,makecell,array}
\usepackage{array,mathrsfs}
\usepackage[all,ps,cmtip]{xy} 
\usepackage[normalem]{ulem}
\usepackage{bm,mathtools}
\usepackage{hyperref}
\usepackage{caption}

\usepackage[toc,page]{appendix}

\usepackage{enumitem}
\setlist[itemize]{noitemsep,nolistsep}
\setlist[enumerate]{noitemsep,nolistsep}
\usepackage{colortbl} 

\allowdisplaybreaks
\let\mathcal\mathscr

\def\Z{{\bf Z}}

\def\F{{\bf F}}
\def\C{{\bf C}}

\def\Q{{\bf Q}}
\def\P{{\bf P}}

\def\U{{\bf U}}

\def\av{abelian variety}
\def\avs{abelian varieties}
\def\ppav{principally polarized abelian variety}

\def\ppavs{principally polarized abelian varieties}

\def\phi{\varphi}

\def\cA{\mathcal{A}}

\def\cM{\mathcal{M}}

\def\cO{\mathcal{O}}

\def\cR{\mathcal{R}}

\def\A{{\mathbb A}}
\def\J{{\mathbb{J}}}
\def\G{\mathbb{G}}

\def\QQ{\mathsf{Q}}

\def\mmu{{\boldsymbol \mu}}

\def\lra{\longrightarrow}
\def\llra{\hbox to 10mm{\rightarrowfill}}
\def\lllra{\hbox to 15mm{\rightarrowfill}}

\def\llla{\hbox to 10mm{\leftarrowfill}}
\def\lllla{\hbox to 15mm{\leftarrowfill}}

\def\hra{\hookrightarrow}
\def\lhra{\ensuremath{\lhook\joinrel\relbar\joinrel\rightarrow}}

\def\isom{\simeq}

\def\tY{\widetilde{Y}}

 \def\vide{\varnothing}

\def\go{\mathfrak o}
\def\gS{\mathfrak S}
\def\gA{\mathfrak A}

\DeclareMathOperator{\isomto}{\stackrel{{}_{\scriptstyle\sim}}{\to}}

\DeclareMathOperator{\isomlra}{\stackrel{{}_{\scriptstyle\sim}}{\lra}}

\DeclareMathOperator{\Alb}{Alb}

\DeclareMathOperator{\AJ}{\mathsf{A\!J}}

\DeclareMathOperator{\Aut}{Aut}

\DeclareMathOperator{\Card}{Card}
\DeclareMathOperator{\CH}{CH}

\DeclareMathOperator{\coker}{Coker}

\DeclareMathOperator{\disc}{disc}
\DeclareMathOperator{\Disc}{Disc}
\def\div{\mathop{\rm div}\nolimits}
\DeclareMathOperator{\End}{End}

\DeclareMathOperator{\GL}{GL}
\DeclareMathOperator{\Fix}{Fix}
\DeclareMathOperator{\Hdg}{Hdg}
\DeclareMathOperator{\Gr}{\mathsf{Gr}}

\DeclareMathOperator{\Id}{Id}

\DeclareMathOperator{\Jac}{Jac}

\DeclareMathOperator{\NS}{NS}
 \DeclareMathOperator{\PSL}{PSL}
\DeclareMathOperator{\PGL}{PGL}
\DeclareMathOperator{\Pic}{Pic}

\DeclareMathOperator{\Sp}{Sp}

\DeclareMathOperator{\Spec}{Spec}

\DeclareMathOperator{\Sing}{Sing}

\DeclareMathOperator{\Sym}{Sym}
\DeclareMathOperator{\SL}{SL}

\DeclareMathOperator{\Tr}{Tr}

\def\llra{\hbox to 10mm{\rightarrowfill}}
\def\lllra{\hbox to 15mm{\rightarrowfill}}

\def\bw#1#2{\textstyle{\bigwedge\hskip-0.9mm^{#1}}\hskip0.2mm{#2}}
\def\sbw#1#2{\small{\bigwedge\hskip-0.9mm^{#1}}\hskip0.2mm{#2}}

\newtheorem{lemm}{Lemma}[section]
\newtheorem{theo}[lemm]{Theorem}
\newtheorem{coro}[lemm]{Corollary}
\newtheorem{prop}[lemm]{Proposition}

\theoremstyle{definition}

\newtheorem{rema}[lemm]{Remark}

\newtheorem{exam}[lemm]{Example}

\theoremstyle{remark}
\newtheorem*{remark*}{Remark}
\newtheorem*{note*}{Note}

\def\moins{\smallsetminus}

\def\hk{{hyperk\"ahler}}

\def\SN{{\mathsf S}}
\def\mon{{Klein}}

\def\setminus{\smallsetminus}

\makeatletter
\@namedef{subjclassname@2020}{%
  \textup{2020} Mathematics Subject Classification}
\makeatother

\begin{document}
\title[GM varieties with many symmetries]{
Gushel--Mukai varieties with many symmetries and an explicit irrational Gushel--Mukai threefold}

\author[O.\ Debarre]{Olivier Debarre}
\thanks{This project has received funding from the European
Research Council (ERC) under the European
Union's Horizon 2020 research and innovation
programme (Project HyperK --- grant agreement 854361).}
\address{Universit\'e de Paris, CNRS, 
IMJ-PRG, F-75013 Paris, France}
 \email{{\tt olivier.debarre@imj-prg.fr}}
 
 \author[G.\ Mongardi]{Giovanni Mongardi}
\address{Dipartimento di Matematica, Universit\`a degli studi di Bologna, Piazza Di Porta San Donato 5, Bologna, Italia 40126}
 \email{{\tt giovanni.mongardi2@unibo.it}}

 \date{\today}


\begin{abstract}
We construct an explicit   complex  smooth Fano threefold with Picard number 1,  index 1, and degree 10 (also known as a Gushel--Mukai threefold) and prove that it is not rational by showing that its intermediate Jacobian has a faithful $\PSL(2,\F_{11}) $-action.\ Along the way, we construct Gushel--Mukai varieties of various dimensions with  rather large (finite) automorphism groups.\ 
The starting point  of all these constructions is an Eisenbud--Popescu--Walter sextic with a faithful $\PSL(2,\F_{11}) $-action discovered  by the second author in 2013.
\end{abstract}

\maketitle

\hfill{\it To Fabrizio Catanese, on the occasion of his 70+1st birthday}

\section{Introduction}

The problem of the rationality of complex unirational smooth Fano threefolds has now been solved in most cases but there are still some unanswered questions.\ For example, Beauville   established in \cite[Theorem.~5.6(ii)]{bea1}, by a degeneration argument using the Clemens--Griffiths criterion,  that a {\em general} Fano threefold with Picard number 1,  index 1, and degree 10 (also known as a Gushel--Mukai, or GM, threefold) is irrational, but not a single smooth example was known, although it is  expected that all of these Fano threefolds are irrational.\ One of the main results of this article is the construction of a complete 2-dimensional family of such examples (Corollary~\ref{coro52}), including one such threefold defined (over~$\Q$) by explicit equations  (Section~\ref{se23}, Corollary~\ref{coro53}).

Our starting point  was a remarkable EPW (for Eisenbud--Popescu--Walter) sextic hypersurface $Y_\A\subset \P^5$, constructed in \cite{monphd}, with a faithful action   by    the 
 simple group $\G:=\PSL(2,\F_{11}) $ of order $660$ (Section~\ref{sect32}).\ We prove that the  automorphism group of~$Y_\A$ is exactly~$\G$ (Proposition~\ref{prop:all_autom_A}) and that it  is the only quasi-smooth EPW sextic with an automorphism of order $11$ (Theorem~\ref{th47}).\

From this sextic, one can construct   GM varieties of various dimensions  with exotic properties.\ Using \cite{dkeven},  we obtain for example families of GM varieties of dimensions $4$ or~$6$ with middle-degree Hodge groups  of maximal rank~22 (Section~\ref{sect46}).\ 

Another  application is the construction of GM varieties with large (finite) automorphism groups.\ The foremost example is 
 a GM fivefold $X^5_\A$ with automorphism group $\G$ (Corollary~\ref{cor48}(2)) but we also construct GM varieties of various dimensions with automorphism groups  $\Z/11\Z$, $D_{12}$, $\Z/6\Z$, $\Z/3\Z$, $D_{10}$, $\Z/5\Z$, $\gA_4$, $(\Z/2\Z)^2$, or $\Z/2\Z$   (Table~\ref{tabaut}).\ 
 
 By \cite{dkij}, the intermediate Jacobians of  the GM varieties of dimension $3$ or $5$ obtained from the sextic $Y_\A$ are all isomorphic to a fixed \ppav\ $(\J,\theta)$ of dimension~$10$.\ This applies in particular to $X^5_\A$, and  the $\G$-action on 
$X^5_\A$ induces a faithful $\G$-action on $(\J,\theta)$.\ We use this fact 
 to prove that  the GM threefolds  that we construct from $Y_\A$ are not rational: by   the  Clemens--Griffiths criterion (\cite[Corollary~3.26]{cg}),   it suffices  to prove that their  (common)  intermediate Jacobian $(\J,\theta)$ is not a product of Jacobians of curves.\ For this, we
  follow  \cite{bea2,bea3} and use the fact that $(\J,\theta)$ has ``too many automorphisms'' (because of the $\G$-action).\ Note that the GM threefolds themselves may have no nontrivial automorphisms.\ This is how we produce a  complete 2-dimensional family of irrational GM threefolds, all mutually birationally isomorphic.

The 10-dimensional   \ppav\ $(\J,\theta)$ seems an interesting object of study.\ The   10-dimensional complex representation attached to the $\G$-action is  irreducible and defined over $\Q$.\ This implies that $(\J,\theta)$ is indecomposable and isogeneous to the  product of $10$ copies of an elliptic curve  (Propositions~\ref{prop61} and~\ref{prop62}).\ We conjecture, but were unable to prove, that~$(\J,\theta)$ is isomorphic to an explicit 10-dimensional \ppav\ that we construct in Proposition~\ref{prop63}.\ 

The situation is reminiscent of that of the Klein cubic threefold $W\subset\P^4$: Klein proved in~\cite{kle} that $W$ has a faithful linear $\G$-action; one hundred years later, Adler proved in \cite{adl} that the automorphism group of $W$ is exactly~$\G$ and Roulleau showed in \cite{rou} that $W$ is the only smooth cubic
threefold with an automorphism of order 11.\ The intermediate Jacobian of~$W$ is a \ppav\ of dimension $5$ isomorphic   to the  product of $5$ copies of an elliptic curve with complex multiplication and Adler proved in \cite{adls} that it is the only \av\ of dimension 5 with a faithful action of $\G$.\ This is the   reason why we call our   sextic $Y_\A$ the Klein EPW sextic.\ We also refer to \cite{cks} for the construction of a one-dimensional family of threefolds with $\gS_6$-actions whose intermediate Jacobians are isogeneous to the  product of $5$ copies of  varying elliptic curves (\cite[Remark~4.5]{cks}).

Our proofs  heavily use the construction by O'Grady in~\cite{og7} of canonical double covers of quasi-smooth EPW sextics called double EPW sextics (see also \cite{dkcovers}).\ They are smooth \hk\ fourfolds whose automorphisms may, thanks to Verbitsky's Torelli Theorem, be determined using  lattice theory.\ We also use the  close relationship between EPW sextics and GM varieties developed in \cite{im,dkclass,dkeven,dkmoduli,dkij} and surveyed in \cite{debsur}.

The article is organized as follows.\ In Section~\ref{sect2}, we recall basic facts about EPW sextics and GM varieties.\ In Section~\ref{se3}, we describe explicitly the  \mon\  Lagrangian $\A$ and the Klein EPW sextic~$Y_\A$, and we prove that the EPW sextic $Y_\A$ is quasi-smooth.\  In Section~\ref{sect4}, we   prove that the automorphism group of $Y_\A$ is $\G$; we also prove that $Y_\A$ is the only quasi-smooth EPW sextic with an automorphism of order~$11$.\  We also discuss the possible automorphism groups and some Hodge groups of the various GM varieties that can be constructed from the Lagrangian~$\A$.\ 
In Section~\ref{sect5}, we introduce the important surface~$\tY_A^{\ge 2}$ (a double \'etale cover of the singular locus of $Y_\A$) and its Albanese variety $(\J,\theta)$.\ We prove our irrationality results for GM threefolds and discuss the structure of the   10-dimensional \ppav\ 
 $(\J,\theta)$.\
 
 The rest of the article consists of appendices.\ In the long
Appendix~\ref{appC}, we gather old and new general results on automorphisms of double EPW sextics and of  double EPW surfaces.\ 
   Appendix~\ref{sea2} recalls a few classical facts about representations of the group $\G$.\ Appendix~\ref{b2} discusses   decomposition results for   abelian varieties with automorphisms.

\noindent{\bf Notation.}  Let $m$ be a positive integer; throughout this article,   $V_m$  denotes a complex vector space of dimension~$m$ and we   set
  $\zeta_m:=e^{\frac{2\pi i}{m}}$.\ As we did above, we  denote by
 $\G$ the simple group $\PSL(2,\F_{11}) $
   of order $660$.\

  \noindent{\bf Acknowledgements.}  We would like to thank   B.~Gross,   G.~Nebe, D.~Prasad, Yu.~Prokhorov, and O.~Wittenberg for fruitful exchanges.\ Special thanks go to A.~Kuznetsov, whose numerous comments and suggestions helped improve the exposition and the results of this article; in particular, Propositions~\ref{split1} and~\ref{split2} are his.

\section{Eisenbud--Popescu--Walter sextics and Gushel--Mukai varieties}\label{sect2}

We recall in this section a few basic facts about Eisenbud--Popescu--Walter (or EPW for short) sextics and Gushel--Mukai (or GM for short) varieties.

\subsection{EPW sextics and their automorphisms}\label{se1}

Let $V_6$ be a $6$-dimensional complex vector space.\ We endow $\bw3V_6$ with the $\bw6V_6$-valued symplectic form defined by wedge product.\
Given a  Lagrangian subspace $A\subset \bw3V_6$ and a nonnegative integer $\ell$, one defines (see \cite[Section 2]{og1} or \cite[Appendix~B]{dkclass}) in   $\P(V_6)$  the closed subschemes
\begin{equation*}\label{yabot}
Y_A^{\ge \ell}:=\bigl\{[x]\in\P(V_6) \mid \dim\bigl(A\cap (x \wedge\bw{2}{V_6} )\bigr)\ge \ell\bigr\}
\end{equation*}
 and the locally closed subschemes
 \begin{equation*}\label{yaell}
 Y_A^\ell :=\bigl\{[x]\in\P(V_6) \mid \dim\bigl(A\cap (x \wedge\bw{2}{V_6} )\bigr)= \ell\bigr\} = Y_A^{\ge \ell} \setminus Y_A^{\ge \ell + 1}.
 \end{equation*}
We  henceforth assume that 
$A$ contains no decomposable vectors (that is, no nonzero products $x\wedge y\wedge z$).\ The scheme  $Y_A:=Y_A^{\ge 1}$ is then an integral sextic hypersurface (called an {\em EPW sextic}) whose singular locus is the integral surface $Y_A^{\ge 2}$; the singular locus of that surface is the finite set $Y_A^{\ge 3}$ (see \cite[Theorem~B.2]{dkclass}) which is empty for $A$ general.\

One
  has moreover (\cite[Proposition~B.9]{dkclass})
\begin{equation}\label{autya}
\Aut(Y_A)=\{ g\in  \PGL(V_6)\mid (\bw3g)(A)=A\}
\end{equation}
and this group is finite.

\subsection{GM varieties and their automorphisms}\label{se22n}

A (smooth ordinary) GM variety of dimension $n\in\{3,4,5\}$ is the smooth complete intersection, in $\P(\bw2V_5)$, of the Grassmannian~$\Gr(2,V_5)$ in its Pl\"ucker embedding, a linear space $\P^{n+4}$, and a quadric.\ It is a Fano variety with Picard number~$1$,  index~$n-2$, and degree~$10$.

There is a bijection between the set of isomorphism classes of (smooth  ordinary) GM varieties~$X$ of dimension $n$ and the set of isomorphism classes  of triples  $(V_6,V_5,A)$, where  $A\subset\bw3 V_6$ is a Lagrangian subspace with no decomposable vectors  and $V_5\subset V_6$ is a hyperplane  such that
    \begin{equation}\label{yperp}
\dim (A\cap \bw3V_5)=5-n
  \end{equation}
(this bijection was first described in the proof of~\cite[Proposition~2.1]{im} when $n=5$;  for the general case, see \cite[Theorem~3.10 and Proposition~3.13(c)]{dkclass} or~\cite[(2)]{debsur}).

By \cite[Lemma~2.29 and Corollary~3.11]{dkclass}, we have
   \begin{equation}\label{autxa}
   \Aut(X)\isom \{ g\in  \Aut(Y_A)\mid g(V_5)=V_5\}.
    \end{equation}

\section{The \mon\  Lagrangian}\label{se3}

The following construction  of an EPW sextic  
with a faithful  $\G$-action first appeared in \cite[Example~4.5.2]{monphd}.\ 

\subsection{The \mon\  Lagrangian $\A$ and the    GM fivefold $X_\A^5$}\label{se31}  

Let $\xi\colon\G\to\GL(V_\xi)$ be the irreducible
 representation of~$ \G$ of dimension 5 described in Appendix~\ref{sea2}.\ From the existence of a unique (up to multiplication by a nonzero scalar) $\G$-equivariant
  symmetric  isomorphism
  \begin{equation}\label{defw}
w\colon \bw2V_\xi\isomlra \bw2V_\xi^\vee 
\end{equation}
  as in~\eqref{defu}, we infer that   there is a unique $\G$-invariant quadric 
  \begin{equation}\label{defq}
\QQ \subset \P(\bw2V_\xi)
\end{equation}
and that it is smooth.\ Since its equation does not lie  in the image of the $\G$-equivariant morphism
$$V_\xi\isom \bw4V_\xi^\vee \lhra \Sym^2(\bw2V_\xi^\vee),
$$
which is the  space of Pl\"ucker quadrics, the quadric $\QQ$ does not contain 
  the Grassmannian~$\Gr(2,V_\xi)$.\ Therefore, it defines a GM fivefold
\begin{equation}\label{defx}
X_\A^5:=\QQ\cap \Gr(2,V_\xi) 
\end{equation}
with a faithful $\G$-action (we will show below that $X_\A^5$ is smooth).\ 

The group $\G$ being simple nonabelian, the representation $\bw5\xi$ is trivial.\ The isomorphism~$w$ from~\eqref{defw} therefore induces  an isomorphism of representations
\begin{equation}\label{defv}
v\colon \bw2V_\xi\isomlra  \bw2V_\xi^\vee \otimes \bw5V_\xi \isomlra \bw3  V_\xi.
\end{equation}
Since $w$ is symmetric, $v$ satisfies
 $v(x)\wedge y=x\wedge v(y)$ for all $x,y\in \bw2V_\xi$.\  

Let $\chi_0\colon \G\to V_{\chi_0}$ be the trivial representation and consider   the $\G$-representation 
$$V_6:=V_{\chi_0}\oplus V_\xi.$$
The    decomposition of~$\bw3V_6$ into  irreducible $\G$-representations is  
\begin{equation}\label{deco}
\bw3V_6=( V_{\chi_0}\wedge  \bw2V_\xi)\oplus \bw3V_\xi 
\end{equation}
and, if $e_0$ is a generator of $V_{\chi_0}$, the Lagrangian subspace $\A\subset \bw3V_6$ associated with the GM fivefold $X^5_\A$ according to the general procedure outlined  in Section~\ref{se22n}
is  the graph
$$\A:=\{ e_0\wedge x+  v(x)\mid x\in  \bw2V_\xi\}$$
of $v$.\
Conversely, $X^5_\A$ is the GM fivefold associated with the Lagrangian $\A$ and the hyperplane \mbox{$V_\xi\subset V_6$} (referring to~\eqref{yperp}, note that $\A\cap \bw3V_\xi=0$).

We will use the following notation.\ Let $c$ and $a$ be the elements of $\G$ defined in Appendix~\ref{sea2}
 and let $(e_1,\dots,e_5)$ be a basis of $V_\xi$ in which $\xi(c)$ and $\xi(a)$   have matrices as in~\eqref{real}.\ Let  $(e^\vee_1,\dots,e^\vee_5)$ be the  dual basis of~$V_\xi^\vee$.\ We also set
 $e_{i_1\cdots i_r}=e_{i_1}\wedge \dots \wedge e_{i_r}\in \bw{r}V_6$.

\begin{prop}\label{prop:GM5_smooth}
The GM fivefold $X^5_\A$ is smooth and the Lagrangian subspace $\A$ contains no decomposable vectors.
\end{prop}

\begin{proof}
  The basis $(e_{ij})_{1\le i<j\le 5}$ of $\bw{2}V_\xi$   consists  of eigenvectors  of $\bw{2}\xi(c)$, with eigenvalues all the primitive $11^{\textnormal{th}}$ roots of $1$, and similarly for the dual basis $(e_{ij}^\vee)_{1\le i<j\le 5}$ of $\bw{2}V_\xi^\vee$.\  
Looking at the corresponding eigenvalues, we see that  we may normalize  the isomorphism $w$ in~\eqref{defw} so that it satisfies $w(e_{12})=-e_{13}^\vee$ (both are eigenvectors of $\bw{2}\xi(c)$ with eigenvalue~$\zeta_{11}^{5}$).\ Applying $\bw{2}\xi(a)$, we find 
$$
w(e_{12})=-e_{13}^\vee,\ w(e_{23})=-e_{24}^\vee,\ w(e_{34})=-e_{35}^\vee,\ w(e_{45})=e_{14}^\vee,\ w(e_{15})=-e_{25}^\vee.
$$
Since $w$ is symmetric, we also have
$$
w(e_{13})= -e_{12}^\vee,\ w(e_{24})=- e_{23}^\vee,\ w(e_{35})=- e_{34}^\vee,\ w(e_{14})= e_{45}^\vee,\ w(e_{25})= -e_{15}^\vee.$$
The quadric $\QQ$ from~\eqref{defq} is therefore  defined by  
\begin{equation}\label{eqQ}
x_{12}x_{13}+x_{23}x_{24}+x_{34}x_{35}-x_{45}x_{14}+x_{15}x_{25}=0.
\end{equation}
A  computer check with \cite{m2}
 now ensures that the GM fivefold $X^5_\A$ defined by~\eqref{defx}  is smooth.\ It   follows from \cite[Theorem~3.16]{dkclass} that $\A$ contains no decomposable vectors.
\end{proof}

The group $\G$  acts faithfully on the  GM fivefold $X^5_\A$.\ Using
  the isomorphism~\eqref{autxa}, we see that it also acts faithfully on the EPW sextic~$Y_\A$ by linear automorphisms that fix the hyperplane $V_\xi$.\ More precisely, the representation $\chi_0\oplus \xi\colon \G\hra \GL(V_6)$ induces an  embedding
$ \G \hra  \Aut(Y_\A)\subset \PGL(V_6)$.\
We will prove in Proposition \ref{prop:all_autom_A} that the embedding
$ \G \hra  \Aut(Y_\A)$ is in fact an isomorphism.

 \subsection{Explicit equations}\label{sect32}
As we saw in the proof of Proposition~\ref{prop:GM5_smooth}, and with the notation of that proof,  the isomorphism $v\colon \bw2V_\xi\isomto \bw3  V_\xi$ from~\eqref{defv} may be defined by
\begin{equation}\label{v2}
\begin{aligned}
v(e_{12})=e_{245},\ v(e_{23})=e_{135},\ v(e_{34})&=e_{124},\ v(e_{45})=e_{235},\ v(e_{15})=-e_{134},\\
v(e_{13})= -e_{345},\ 
v(e_{24})= -e_{145},\ v(e_{35})&=- e_{125},\ v(e_{14})= e_{123},\ v(e_{25})=e_{234}.
\end{aligned}
\end{equation}
This gives  
  \begin{equation}\label{defA}
  \begin{aligned}
\A=  \langle &
e_{012}+  e_{245},
 e_{013} -  e_{345},
 e_{014} +  e_{123}, 
  e_{015} -    e_{134},
  e_{023} +  e_{135},\\
   &\qquad e_{024}- e_{145} , 
  e_{025} + e_{234} , 
 e_{034}+ e_{124} ,
 e_{035}- e_{125} , 
 e_{045}+ e_{235} \rangle.
\end{aligned}
\end{equation}
One can readily see from this that the isomorphism $V_6\isomto V_6^\vee$ that sends $e_0$ to $-e_0^\vee$ and $e_j$ to~$e_j^\vee$ for   $j\in\{1,\dots,5\}$ maps $\A$ onto its orthogonal $\A^\bot$, a Lagrangian subspace of $\bw3V_6^\vee$; we say that $\A$ is {\em self-dual.}\ 
Also, if one starts from the dual representation~$\xi^\vee$, one obtains the same Lagrangian~$\A$.

\begin{prop}\label{yaqs}
The EPW sextic~$Y_\A$ is defined by the equation
 \begin{equation}\label{sextic_equation}
\begin{aligned}
&x_0^6+ 2x_0^3(x_1x_3^2+x_2x_4^2+x_3x_5^2+x_4x_1^2+x_5x_2^2)-4x_0(x_1^3x_2^2+ x_2^3x_3^2+ x_3^3x_4^2+x_4^3x_5^2+x_5^3x_1^2)\\
&{}+4x_0(x_1x_3x_4^3+x_2x_4x_5^3+x_3x_5x_1^3+ x_4x_1x_2^3 + x_5x_2x_3^3)-12x_0x_1x_2x_3x_4x_5\\
&{}+x_1^2x_3^4+x_2^2x_4^4 +x_3^2x_5^4+x_4^2x_1^4+x_5^2x_2^4 -4(x_1x_4x_5^4+x_2x_5x_1^4+x_3x_1x_2^4+x_4x_2x_3^4+x_5x_3x_4^4) \\
&{}-2(x_1x_3^3x_5^2+x_2x_4^3x_1^2+x_3x_5^3x_2^2+x_4x_1^3x_3^2+x_5x_2^3x_4^2)\\
&{}+6(x_1x_2x_3^2x_4^2+x_2x_3x_4^2x_5^2+x_3x_4x_5^2x_1^2+x_4x_5 x_1^2x_2^2 + x_5x_1x_2^2x_3^2)=0
\end{aligned}
\end{equation}
  in $\P(V_6)$.\
The   scheme $ Y^{\ge 2}_\A$ is a smooth irreducible surface, so that the scheme $Y^{\ge 3}_\A$ is empty.\ 
\end{prop}

 \begin{proof} The scheme $Y_\A$ is the locus in $\P(V_6)$ where the map
$$   x\wedge \bw2 V_6\lra \bw3 V_6/\A$$
drops rank.\ In the decomposition~\eqref{deco}, the second summand is transverse to $\A$ and we can identify $\bw3 V_6/A$ with $\bw3 V_\xi$.\ Moreover, in the affine open subset $U_0$ of $\P(V_6)$ defined by $x_0\neq 0$, one has $x\wedge \bw2 V_6=x\wedge \bw2 V_\xi$.\ In $U_0$, the scheme $Y_\A$ is therefore the locus where the map
$$   x\wedge \bw2 V_\xi\lra \bw3 V_\xi\xrightarrow{\ v^{-1}\ } \bw2 V_\xi$$
drops rank.\ Concretely,  if $x=e_0+x_1e_1+\dots+x_5e_5$, we  see, using~\eqref{defA} and~\eqref{v2}, that it maps
\begin{equation*}
\begin{aligned}
e_{12}&\longmapsto x\wedge e_{12}=e_{012}+x_3e_{123}+x_4e_{124}+x_5e_{125}\\
&\longmapsto -e_{245}+x_3e_{123}+x_4e_{124}+x_5e_{125}\\
&\longmapsto -e_{12}+x_3e_{14}+x_4e_{34}-x_5e_{35}.
\end{aligned}
\end{equation*}
All in all, using the basis $(e_{12},e_{13},e_{14},e_{15},e_{23},e_{24},e_{25},e_{34},e_{35},e_{45})$ of $\bw2V_\xi$, one sees that $Y_\A\cap U_0$ is defined  as the determinant of the $10\times 10$ matrix
 \begin{equation*}\label{matrixA}
\left(
\begin{smallmatrix}
-1&0&0&0&0&x_5&-x_4&0&0&x_2\\
0&-1&0&0&0&0&0&-x_5&x_4&-x_3\\
x_3&-x_2&-1&0&x_1&0&0&0&0&0\\
0&-x_4&x_3&-1&0&0&0&-x_1&0&0\\
0&x_5&0&-x_3&-1&0&0&0&x_1&0\\
0&0&-x_5&x_4&0&-1&0&0&0&-x_1\\
0&0&0&0&x_4&-x_3&-1&x_2&0&0\\
x_4&0&-x_2&0&0&x_1&0&-1&0&0\\
-x_5&0&0&x_2&0&0&-x_1&0&-1&0\\
0&0&0&0&x_5&0&-x_3&0&x_2&-1
\end{smallmatrix}\right).
\end{equation*}
 We obtain the  equation~\eqref{sextic_equation} by homogenizing this determinant, computed with Macaulay2 (\cite{m2}).\ We then check with Macaulay2 that $\Sing(Y_\A)$ is a smooth   surface (this reproves that $\A$ contains no decomposable vectors and proves in addition that $Y^{\ge 3}_\A$ is empty).
\end{proof}

 
\subsection{The    GM threefold $X_\A^3$}\label{se23}
We keep the notation above.\  By Proposition \ref{yaqs},  $Y^{\ge 3}_\A $ is empty and, since $\A$ is self-dual, so is $Y^{\ge 3}_{\A^\bot}$.\ For all hyperplanes $V_5\subset V_6$, we thus have
\begin{equation}\label{y3vide}
\dim (\A\cap \bw3V_5)\le 2.
\end{equation}
Consider  the hyperplane $V_5\subset V_6$ spanned by $e_0,\dots,e_4$.\ 
From the description~\eqref{defA}, one sees that there is an inclusion
\begin{equation*}
\langle e_{014} +  e_{123}, e_{034}+ e_{124}\rangle \subset \A\cap \bw3 V_5
\end{equation*}
of vector spaces which, because of the inequality~\eqref{y3vide}, is an equality.\ The associated 
 GM variety is therefore smooth of dimension~$3$ (see Section~\ref{se22n}).\ Using the automorphism $\xi(a)$ of $V_6$ that permutes the vectors $e_1,\dots, e_5$, we see that we get   isomorphic GM threefolds if we start from   hyperplanes spanned by $e_0$ and any four vectors among $e_1,\dots, e_5$.\ We denote it by~$X^3_\A$.

Going through the procedure mentioned in Section~\ref{se22n}, A.~Kuznetsov found that $X^3_\A$ is the intersection, in $\P(\bw2 V_5)$,  of the Grassmannian $\Gr(2,V_5)$, the linear space $\P^7$ with equations
$$
x_{03} + x_{12} = x_{04} - x_{23} = 0,
$$
and the quadric with equation
$$
x_{01}x_{02} - x_{13}x_{14} - x_{24}x_{34} = 0.
$$
 
\section{EPW sextics and GM varieties with many automorphisms}\label{sect4}

As in Section~\ref{se1}, let $V_6$ be a $6$-dimensional complex vector space and let $A\subset \bw3V_6$ be a  Lagrangian subspace with no decomposable vectors.\ It defines an integral EPW sextic $Y_A\subset \P(V_6)$.\ As explained in more detail in Appendix~\ref{se41}, there is a canonical double covering $\pi_A\colon \tY_A\to Y_A
$ and, when $Y^{\ge 3}_A=\vide$, the fourfold $\tY_A$ is  a  smooth \hk\ variety of K3$^{[2]}$-type.\

\subsection{Automorphisms of the EPW sextic $Y_\A$}\label{sec43}
We constructed at the end of Section~\ref{se31} an  injection $\G\hra \Aut(Y_\A)$.\  It follows from Proposition~\ref{yaqs}  that   the double EPW sextic $\tY_\A$ is smooth and, by Proposition~\ref{split1}, the group $\Aut(Y_\A)$ is isomorphic to the group~$\Aut_H^s(\tY_\A) $ of symplectic isomorphisms of $\tY_\A$ that preserve the polarization $H$.\

\begin{prop}\label{prop:all_autom_A}
The automorphism group of the Klein  EPW sextic $Y_\A$ is isomorphic to $\G$.
\end{prop}

\begin{proof}
It is enough to prove that   $  \Aut^s_H(\tY_\A) $ is isomorphic to $\G$.\ Let $g\in\Aut^s_H(\tY_\A)$.\ 
It acts on the orthogonal of $H$ in $\Pic(\tY_\A )$ which, by Corollary \ref{th14}, is the rank-20 lattice~$\SN $ discussed in Section~\ref{secc1} and the action  is faithful.\ Let   us prove that $g$ acts trivially on the discriminant group~$\Disc(\SN)$.\ 

By Corollary~\ref{th14}, the lattice $H^\perp\isom (-2)^{\oplus 2}\oplus E_8(-1)^{\oplus 2}\oplus U^{\oplus 2}\subset H^2(\tY_\A,\Z)$ (see~\eqref{defhperp}) primitively contains the lattices $\Tr(\tY_\A)\isom (22)^{\oplus 2}$ and $\SN$ and it is a finite   extension of their direct sum.\ This extension is obtained by adding to $\Tr(\tY_\A)\oplus\SN$ two elements $\frac{a_1+b_1}{11}$ and $\frac{a_2+b_2}{11}$, where~$a_1$ and~$a_2$ are  orthogonal generators of $\Tr(\tY_\A)$ of square $22$, and  $b_1$ and~$b_2$ are classes in~$\SN$ of divisibility 11.\ Since $g$ preserves $H^\perp$ and  $\Tr(\tY_\A)$, it follows readily  that $g(b_i)=b_i+11c_i$ for some $c_i\in\SN$, which implies that $g$ acts trivially on $\Disc(\SN)$, as claimed.

The proposition follows since, by \cite[Table 1, line 120]{HM}, the  group of isometries of $\SN$ that act trivially on  $\Disc(\SN)$ coincides with $\G$.\  
\end{proof}

\subsection{GM varieties with many symmetries}\label{sec42n}

Proposition~\ref{prop:all_autom_A} can be used to determine the automorphism groups of the GM varieties constructed from the Lagrangian $\A$, and in particular the varieties~$X^5_\A$ and $X^3_\A$ defined in Sections~\ref{se31} and~\ref{se23}.\ By~\eqref{autxa}, all we have to do is determine the stabilizers of hyperplanes in $V_6$ under the $\G$-action.\ Since this action is conjugate to its dual, we might as well determine the stabilizers of  lines in $V_6=\C e_0\oplus V_\xi$.\ We    proceed in three steps:
\begin{itemize}
\item   determine  the various fixed-point sets of all subgroups of $\G$, listed up to conjugacy in \cite[Figure~1]{bue};
\item  compute the stabilizers of these fixed-points;
\item  find in which stratum $Y_\A^\ell$ they lie.
\end{itemize}

A first useful remark is the following: {\em if $g\in\G$ is a nontrivial element of odd order, the   fixed-point set of $g$ in~$Y_\A$ is finite.}\ Indeed, we will see below by a case-by-case analysis that 
the fixed-point set $\Fix(g)$ of $g$ in $\P(V_6)$ is a union of lines and isolated points.\ Assume that a line $\Delta\subset \Fix(g)$ is contained in $Y_\A$.\
By Proposition~\ref{split1}, $g$ lifts to a symplectic automorphism~$\tilde g$ of~$\tY_\A$ which commutes with its covering involution $\iota$.\ For any $x$ in the curve $ \pi_\A^{-1}(\Delta)\subset \tY_\A$, one has either $\tilde g(x)=x$ or $\tilde g(x)=\iota(x)$, hence $\tilde g^2(x)=x$.\ The curve 
$ \pi_\A^{-1}(\Delta)\subset \tY_A$ is therefore contained in the fixed-point set of the nontrivial symplectic automorphism $\tilde g^2$.\ But this fixed-point set is, on the one hand, a disjoint union of surfaces and isolated points and, on the other hand, contained in $ \pi_\A^{-1}(\Fix(g^2))$, whose dimension is at most $1$ (because~$g^2$ is again nontrivial of odd order), so we reach a contradiction.\ Moreover, $1$ is not an eigenvalue for the action of $g$ on the tangent space at a fixed-point, hence any line in $\Fix(g)$ meets $Y^1_\A$ and $Y^2_\A$ transversely.

Furthermore, since $g$ itself can be written as a square, we see that the fixed-point set of its symplectic lift $\tilde g$  (which has the same order) is the inverse image in $\tY_A$ of $\Fix(g)$.

Our second tool will be the Lefschetz topological fixed-point theorem for an automorphism~$g$  {\em with finite fixed-point set} on the regular surface $Y_\A^{ \ge2}$.\ This theorem  reads
\begin{equation*}
\#(\Fix(g)\cap Y_\A^{ \ge2})=\sum_{i=0}^4(-1)^i\Tr (g^*\vert_{H^i(Y_\A^{ \ge2},\Q)})=2+\Tr (g^*\vert_{H^2(Y_\A^{ \ge2},\Q)}).
 \end{equation*}
 The group $ \G$ acts on $\A$ (via the representation $\bw2 V_\xi$) and $Y_A^{ \ge2}$ and, by Proposition~\ref{propc7}, the isomorphism $H^2(Y_\A^{ \ge2},\C)\isom \bw2(\A\oplus \bar \A)$ from~\eqref{h1} is equivariant for these actions.\ Using the fact that the representation $\bw2 V_\xi$ is self-dual and  the formula
$$\chi_{\sbw2(\sbw2 V_\xi\oplus \sbw2 V_\xi)}(g)
=2\chi_{\sbw2(\sbw2 V_\xi)}(g)+\chi_{  \sbw2V_\xi\otimes   \sbw2V_\xi}(g)=2\chi_{ \sbw2 V_\xi}(g)^2-\chi_{  \sbw2V_\xi}(g^2),
$$
one can then compute the  numbers of fixed points of $g$ in $Y_\A^{ \ge2}$ given in Table~\ref{tabf}.

The Lefschetz theorem was also used to the same effect in \cite[Section~6.2]{monphd} on \hk\ varieties of  K3$^{[2]}$-type.\ It gives, for symplectic automorphisms of $\tY_\A$ of prime order, the  number (when finite) of fixed-points on $\tY_\A$.\ By the remark made above, this is the number of fixed points on $Y^{\ge 2}_\A$ (which we get from  Table~\ref{tabf}) plus  twice the number of fixed points on~$Y^1_\A$.\ So we get from \cite[Section~6.2]{monphd} the following numbers (except for the information between parentheses (when~$g$ has order $2$ or $6$), which will be a consequence of the discussion below---where it will not be used).
          \begin{table}[h]
\renewcommand\arraystretch{1.5}
 \begin{tabular}{|c|c|c|c|c|c|c|c|c|}
 \hline
order of $g$&$11$&$5$&$6$&$3$&$2$
\\
 \hline
$\# (\Fix (g)\cap Y^{ \ge2}_\A)$
 &$5$&$2 $&$ 3$ &$3 $ & $(\dim 1)$ 
\\
 \hline
$\# (\Fix (g)\cap Y_\A)$
 &$5$&$8 $&$(7) $&$ 15$& $(\dim 2)$ 
\\
 \hline
 \end{tabular}
\vspace{5mm}
\captionsetup{justification=centering} 
\caption{Number (when finite) of fixed-points on\\ the surface $Y^{ \ge2}_\A$ and  the fourfold $Y_\A$}\label{tabf}
\end{table}
\vspace{-5mm}

We will see in the discussion below that these sets   are in fact always finite, except when~$g$ has order 2.\ We can now go through the list of  all subgroups of $\G$ from \cite[Figure~1]{bue} and determine  their various fixed-point sets.\ We will use the notation and results of Appendix~\ref{sea2}.

 \subsubsection{The subgroups $\G$ and $ \Z/11\Z\rtimes \Z/5\Z$}\label{sec421} The subgroups $ \Z/11\Z\rtimes \Z/5\Z$ of $\G$ are all conjugate to the subgroup  generated by the elements $a$ and  $c$ of $\G$.\
 We see from~\eqref{real} that their   only fixed-point   is~$[e_0]$.\ It is on $Y_\A^0$ hence defines a GM fivefold,~$X^5_\A$, already defined in Section~\ref{se31}, with automorphism group $\G$.
 
  \subsubsection{The  subgroups $ \Z/11\Z$}\label{sec422} The subgroups $ \Z/11\Z $ of $\G$ are all conjugate to the subgroup  generated by the element  $c$ of $\G$.\
We see from~\eqref{real} that there are $6$  fixed-points: the point~$[e_0]$ (on~$Y_\A^0$) and 5 other points.\  For these $5$  points, which are all in the same~$\G$-orbit, the stabilizers are exactly $ \Z/11\Z$ (because the only nontrivial oversubgroups are  $ \Z/11\Z\rtimes \Z/5\Z$ and $\G$).\ Furthermore, using Table~\ref{tabf}, one sees that they are in $Y_\A^2$ (this was already observed in Section~\ref{se23}).\
  So we get isomorphic GM threefolds,~$X^3_\A$, already defined in Section~\ref{se23},  with  automorphism groups~$ \Z/11\Z$.

 \subsubsection{The  subgroups $ \Z/3\Z$, $ \Z/6\Z$, and $D_{12}$}\label{sec423}
 The elements of order 6 of $\G$ are all conjugate to the element  $b$ of $\G$.\
Since its character in the representation $\xi$ is $1$, it acts on $V_\xi$ with eigenvalues $1,\zeta_6,\zeta_6^2,\zeta_6^4,\zeta^5_6$, for which we choose eigenvectors $w_0,w_1,w_2,w_4,w_5$.\ The fixed-point set of $b$   consists of the line~$\Delta_6=\langle [e_0],[w_0]\rangle$   and the $4$ isolated points  $[w_1]$, $[w_2]$, $[w_4]$, $[w_5]$.\ Any involution $\tau$ in $\G$ that, together with~$b$, generates a dihedral group $D_{12}$, exchanges the eigenspaces corresponding to conjugate eigenvalues.\ Looking at the subgroup pattern of $\G$, one sees that the stabilizers of the $4$ isolated points are~$ \Z/6\Z$, whereas those of points of $\Delta_6\moins\{[e_0]\}$ are $D_{12}$ (a maximal proper subgroup).\ 

The fixed-point set of an element of $\G$ of order 3 (such as~$b^2$; they are all conjugate) is the union of~$\Delta_6$ and two other disjoint lines,~$\Delta_3=\langle [w_1],[w_4]\rangle$ and~$\Delta'_3=\tau(\Delta_3)=\langle [w_2],[w_5]\rangle$.\ 
The fixed-point set of the subgroup $D_{6}=\langle b^2,\tau\rangle$ is therefore the line $\Delta_6$.\  

Consider now the isomorphism of representations $v\colon \bw2V_\xi\isomto \bw3  V_\xi$ from~\eqref{defv}.\ Looking at the eigenspaces for the action of $b$, we see that we can write
$$ 
v(w_0\wedge w_2)=\alpha w_1\wedge w_2\wedge w_5 
$$
for some $\alpha\in \C$.\ By definition of $\A$, this implies $w_2\wedge (e_0\wedge w_0-\alpha w_1\wedge w_5)\in \A$.\ Similarly, one can write
$$ 
v(w_2\wedge w_5)= \beta   w_1\wedge w_2\wedge w_4+\gamma w_0\wedge w_2\wedge w_5,
$$
for some $\beta,\gamma\in \C $, so that $w_2\wedge (e_0\wedge w_5+\beta   w_1 \wedge w_4+\gamma w_0 \wedge w_5)\in \A$.\ This proves that $[w_2]$ is in~$Y^{\ge2}_\A$, and so is 
$[w_4]=\tau([w_2])$.

 Consider the length-$18$ scheme $\Fix(g^2)\cap Y_\A=Y_\A\cap (\Delta_6\cup \Delta_3\cup \Delta'_3)$.\ We see from  Table~\ref{tabf}   that it has  15 points, 3 of them in $Y^{\ge2}_\A$ (hence nonreduced) and fixed by $g$, therefore $12$ of them in $Y^1_\A $ (reduced by the remark made above), none fixed by $g$.\ 
 Since the set $\Fix(g^2)\cap Y^{\ge2}_\A$ is $\tau$-invariant and contains $[w_2]$ and $ [w_4]$, 
 and
 $g$ acts as an involution with no fixed-points on the set $\Fix(g^2)\cap Y^1_\A\cap \Delta_3$, whose cardinality is thus even, we see that 
 each line  $\Delta_6$, $ \Delta_3$, $ \Delta'_3$ contains a single point of $Y^{\ge2}_\A$ and $4$ points of~$Y^1_\A$;   the   points $[w_1]$ and  $[w_5]$  are in $Y^{0}_\A$.\ In particular, the set $ \Fix (g)\cap Y_\A$ has 7 points, as claimed in Table~\ref{tabf}.

 So altogether, we get GM varieties of dimensions $3$, $4$, or $5$,    with  automorphism groups $\Z/3\Z$,  of dimensions $3$  or $5$   with  automorphism groups~$\Z/6\Z$, and  of dimensions $3$ or $4$   with  automorphism groups $D_{12}$, and we see that  no GM varieties $X_{\A,V_5}$ have 
 automorphism groups the dihedral group~$D_{6}$ or the alternating group~$\gA_5$.
 
\subsubsection{The  subgroups $ \Z/5\Z$ and $D_{10}$}  The subgroups $ \Z/5\Z $ of $\G$ are all conjugate to the subgroup  generated by the element  $a$ of $\G$.\
Since its character is $0$, it acts on $V_\xi$ with eigenvalues $1,\zeta_5,\zeta_5^2,\zeta_5^3,\zeta_5^4$.\ Its fixed-point set in $\P(V_6)$ therefore consists of a line~$\Delta_5$ passing through~$[e_0]$  and~$4$ isolated  points.\ Any involution $\tau$ in $\G$ that, together with~$a$, generates a dihedral group~$D_{10}$, exchanges the eigenspaces corresponding to conjugate eigenvalues.\ Looking at the subgroup pattern of $\G$, one sees that the stabilizers of the $4$ isolated points are~$ \Z/5\Z$, whereas those of points of~$\Delta_5$ contain $D_{10}$.\ Since we saw above that $\gA_5$-stabilizers are not possible, the stabilizers are therefore ~$D_{10}$ for all points of $\Delta_5\moins\{[e_0]\}$.

 Since $\# (\Fix (g)\cap Y_\A)=8$ (Table~\ref{tabf}), one sees that the line $\Delta_5$ meets $Y_\A$ in only 4 points.\ Since $Y^1_\A\cap \Delta_5$ is reduced, at least one of them must be in~$Y_\A^{\ge2}$.\ Among the $4$ isolated  fixed-points, the involution $\tau$ acts with no fixed-points on    the set of those that are in $Y_\A^{\ge2}$, hence its cardinality is even.\ Since $\# (\Fix (g)\cap Y_\A^{\ge2})=2$ (Table~\ref{tabf}), the only possibility is that $\Delta_5$ contain~$2$ points in~$Y_\A^1$ and 2 points in~$Y_\A^2$, and the $4$ isolated points are in~$Y_\A^1$.\ So altogether, we get GM fourfolds with  automorphism groups~$\Z/5\Z$ and 
 GM varieties of dimensions 3, 4, or~$5$ with  automorphism groups $D_{10}$.\

 \subsubsection{The  subgroups $ \Z/2\Z$, $ (\Z/2\Z)^2$, and $\gA_4$} Since its character is $1$, any order-2 element $g$ of~$\G$  acts on $V_\xi$ with eigenvalues $1,1,1,-1,-1$.\ Its fixed-point set in $\P(V_6)$ therefore consists of the disjoint union of a 3-space $\P(V_4)$ passing through $[e_0]$ and a  line $\Delta_2$.\ Double EPW sextics  with a symplectic involution were studied in \cite[Theorem~5]{cam} and \cite[Theorem~6.2.3]{monphd}: they prove that the fixed-point set is always the union of a smooth K3 surface and $28$ isolated points.\ By \cite[Proposition~17]{cam}   (which holds under some generality assumptions  which are satisfied by~$\A$ because it contains no decomposable vectors), we obtain:
 \begin{itemize}
\item $\Fix(g)\cap Y_\A$ is the union of a smooth quadric $Q$ and a Kummer quartic $S$, both contained in $\P(V_4)$, and the 6 distinct points of $Y_\A\cap \Delta_2$;
\item $\Fix(g)\cap Y_\A^{\ge2}$ is contained in $E_2$ and is the disjoint union of the  smooth curve $Q\cap  S$ and the $16$ singular points of $S$.\end{itemize}
The fixed K3 surface in $\tY_\A$ mentioned above is a double cover of $Q$ branched along $Q\cap  S$.\ The   images in $Y_\A$ of the  $28$ fixed-points are the $6$ points of $Y_\A\cap \Delta_2$ and the $16$ singular points of $S$.

The fixed-point set of any subgroup $ (\Z/2\Z)^2$ of $\G$ is a plane~$\Pi_4$ passing through $[e_0]$ and~$3$ isolated points.\ This plane is contained in $\P(V_4)$ and  contains the line $\Delta_6$ fixed by any $D_{12}$ containing $ (\Z/2\Z)^2$.\ For points in $\Pi_4\moins \Delta_6$ and the $3$ isolated points, the stabilizers are either $ (\Z/2\Z)^2$ or~$\gA_4$.\ 

 As an $\gA_4$-representation, $V_6$ splits as the direct    sum of the  $3$ characters (which span the plane~$\Pi_4$) and the one irreducible representation of dimension~$3$.\ It follows that the 
  fixed-point set of any 
 $\gA_4$ containing $ (\Z/2\Z)^2$ has $3$ points (corresponding to  the  $3$ characters), all in $\Pi_4$.\ One of them is~$[e_0]$ and the stabilizer of the other two is indeed $\gA_4$.

 The plane $\Pi_4$ meets $Y_\A$ along the union of the  conic $\Pi_4\cap Q$ and the  quartic curve $\Pi_4\cap S$.\ Since the $1$-dimensional part of $\Fix(g)\cap Y_A^{\ge2}$ is the  smooth octic curve   $Q\cap S$, its intersection  with the plane~$\Pi_4$ is finite nonempty.\ So we get points in $\Pi_4\moins \Delta_6$ (with stabilizers $(\Z/2\Z)^2$) in each of the strata.

Finally, the two fixed-points of~$\gA_4$ are not in $Y_\A$: if they were, we would obtain a point of $\tY_\A$ fixed by a symplectic action of $\gA_4$; however there are no representations of $\gA_4$ in $\Sp(\C^4)$ without trivial summands, so there are no points in $\tY_\A$ fixed by $\gA_4$. Therefore, we only get GM varieties of dimension~$5$ with automorphism groups~$\gA_4$.

 We sum up our results in a table:
           \begin{table}[h]
\renewcommand\arraystretch{1.5}
 \begin{tabular}{|c|c|c|c|c|c|c|c|c|c|c|c|c|}
 \hline
aut. groups&$\G$&$\Z/11\Z$&$D_{12}$&$\Z/6\Z$&$\Z/3\Z$&$D_{10}$&$\Z/5\Z$&$\gA_4$&$(\Z/2\Z)^2$&$\Z/2\Z$&$\{1\}$
\\
 \hline
$\dim(X_{\A,V_5})$
 &$5$& $3 $& $ 3$, $4$ & $ 3$,  $5$ &$ 3$, $ 4$, $5$&$ 3$, $4$, $5$&  $ 4$& $5$& $ 3$, $4$, $5$& $ 3$, $4$, $5$&$ 3$, $4$, $5$
\\
 \hline
 \end{tabular}
\vspace{5mm}
\captionsetup{justification=centering}
\caption{Possible automorphisms groups of (ordinary) GM varieties associated with the  Lagrangian $\A$}
\label{tabaut}
\end{table}

\vspace{-5mm}

\subsection{EPW sextics with an automorphism of order $11$}\label{sect44}

We  use the injectivity of the period map~\eqref{defp} to characterize quasi-smooth EPW sextics with an automorphism of prime order at least~$11$. 

\begin{theo}\label{th47}
The only quasi-smooth EPW sextic with an automorphism of prime order~$p\ge11$ is the EPW  sextic $Y_\A$, and $p=11$. 
\end{theo}

\begin{proof} 
Let $Y_A$ be a quasi-smooth EPW sextic with an automorphism $g$ of prime order~$p\ge11$.\
By Proposition~\ref{split1}, $g$   lifts  to a symplectic automorphism of the same order  of the smooth double EPW sextic $\tY_A$  which fixes the polarization~$H$.\ By   Corollary~\ref{th14}, the transcendental lattice $\Tr(\tY_A)$ is isomorphic to the lattice $T:=  (22)^{\oplus 2}$ and is  primitively embedded in the lattice~$H^\bot$, with orthogonal complement isomorphic to $\SN$.

 \begin{lemm}
Any two primitive embeddings of $T$ into the lattice $h^\bot$ with orthogonal complements isomorphic to $\SN$ differ by an isometry in~$\widetilde O(h^\bot)$.
\end{lemm}

\begin{proof}
 According to~\cite[Proposition~1.5.1]{nik} (see also \cite[Proposition~2.7]{bcs}), 
to primitively embed the lattice $T$ into the lattice $h^\bot$, one needs   subgroups $K_T\subset \Disc(T)\isom  (\Z/22\Z)^2$ and $K_{h^\bot}\subset \Disc(h^\bot)\isom (\Z/2\Z)^2$ and 
an isometry $u\colon K_T\isomto K_{h^\bot}$    for
  the canonical $\Q/2\Z$-valued quadratic forms on these groups.\ The discriminant   of the orthogonal complement is then \mbox{$22^2\cdot 2^2/\Card(K_T)^2$.}\ 
  
  In our case, we want  this orthogonal complement to be $\SN$, with discriminant group $(\Z/11\Z)^2$.\ The only choice is therefore to take $K_T$ to be the $2$-torsion part of $\Disc(T)$ and $K_{h^\bot}=\Disc(h^\bot)$.\ There are only two choices for $u$ and they correspond to switching the two factors of $(\Z/2\Z)^2$.\ Any two such embeddings $T\hra h^\bot$ therefore differ by an isometry of $h^\bot$ and, upon composing with the involution of $h^\bot$ that switches the two  $(-2)$-factors, we may assume that this isometry is in~$\widetilde O(h^\bot)$.
\end{proof}

If we fix any embedding $T\hra h^\bot$ as in the lemma, the period of $\tY_A$ therefore belongs to the (uniquely defined) image in the quotient $ \widetilde O(h^\bot)\backslash \Omega_{h}$ of the set
$ \P(T\otimes \C)\cap \Omega_h$.\ This set consists of   two conjugate points,
one on each component of $\Omega_h$, hence they are mapped to the same point in the period domain $ \widetilde O(h^\bot)\backslash \Omega_{h}$.\ The theorem now follows from the injectivity of the polarized period map, which implies that~$\tY_A$ and $\tY_\A$ are isomorphic by an isomorphism that respects the polarizations.\ Since these polarizations define the double covers $\pi_A$ and $\pi_\A$, this isomorphism descends to an isomorphism between~$Y_A$ and $Y_\A$.
\end{proof}

\begin{coro}\label{cor48}
{\rm(1)} The only smooth double EPW sextic with a symplectic automorphism of prime order~$p\ge11$ fixing the polarization $H$ is the \mon\ double sextic $\tY_\A$, and $p=11$. 

{\rm(2)} The only (smooth ordinary) GM varieties with an automorphism of prime order~$p\ge11$ are the  GM varieties $X_\A^3$ and $X_\A^5$,  and $p=11$.
\end{coro}

\begin{proof}
Part (1) is only a rephrasing of Theorem~\ref{th47}, using the isomorphism $  \Aut_H^s(\tY_A)\isomto \Aut(Y_A)$ from Proposition~\ref{split1}.\ For part (2), let $X$ be a (smooth ordinary) GM variety with an automorphism of prime order~$p\ge11$ and let $A$ be an associated Lagrangian.\ By \eqref{autxa}, the quasi-smooth EPW sextic $Y_A$ also has an automorphism  of  order~$p$.\ It follows from Theorem~\ref{th47} that we can take $A=\A$ and that $p=11$.\ The result now follows from Sections~\ref{sec421} and~\ref{sec422}.
\end{proof}

 \subsection{GM varieties of dimensions 4 and 6 with many Hodge classes}\label{sect46}
  GM sixfolds do not appear in the definition given in Section~\ref{se22n}.\ This is because they are {\em special} (as opposed to {\em ordinary}): they are   double covers $\gamma\colon X\to \Gr(2,V_5)$ branched along the smooth intersection of  $\Gr(2,V_5)$ with a quadric (a GM fivefold!).\  To the GM fivefold correspond a Lagrangian $A$ and a hyperplane $V_5\subset V_6$ such that $A\cap \bw3V_5=\{0\}$.\ When $X$ is a GM fourfold, we let 
  $\gamma\colon X\to \Gr(2,V_5)$ be the inclusion (in both cases, $\gamma$ is called the Gushel map in \cite{dkclass}).
  
 One can use the results of~\cite{dkeven} to construct explicit GM varieties $X$ of even dimensions~$2m\in\{4,6\}$ with   groups $\Hdg^m(X):=H^{m,m}(X)\cap H^{2m}(X,\Z)$ of Hodge classes of maximal rank $h^{m,m}(X)=22$ (\cite[Proposition~3.1]{dkeven}).\ 
  The main  ingredient is~\cite[Theorem~5.1]{dkeven}: there is an isomorphism 
$$(H^{2m}(X,\Z)_{00},\smile)\isom (H^2(\tY_A,\Z)_0,(-1)^{ m-1}q_{BB})
$$
of polarized Hodge structures, where  
$$H^{2m}(X,\Z)_{00}:=\gamma^*H^{2m}(\Gr(2,V_5),\Z)^\bot\subset H^{2m}(X,\Z)$$ and $H^2(\tY_A,\Z)_0$ is, in our previous notation,  $H^\bot\subset H^2(\tY_A,\Z)$.\

If we start from the  Lagrangian $\A$ and any hyperplane $V_5\subset V_6$ that satisfies the condition $\dim(A\cap \bw3V_5)=3-m$, we obtain, by Corollary~\ref{th14}, a family  (parametrized by the fourfold~$Y_\A^1 $ when $m=2$ and by the fivefold~$\P(V_6)\moins Y_\A$ when $m=3$) of GM $2m$-folds $X$ that satisfy  
$$
\begin{aligned}
\Hdg^m(X)( (-1)^{m-1})&\isom \gamma^* H^{2m}(\Gr(2,V_5),\Z)( (-1)^{m-1})\oplus \SN\\
&\isom (2)^{\oplus 2}
\oplus \SN \\
&\isom (2)^{\oplus 2}\oplus E_8( -1)^{\oplus 2}\oplus \begin{pmatrix}
-2 &-1   \\
-1  &   -6
\end{pmatrix}^{\oplus 2}\!\!\!\! ,
\end{aligned}$$
a rank-$ 22$ lattice, the maximal possible rank (the last isomorphism follows from the last isomorphism in the statement of Corollary~\ref{th14}).\ Indeed, $\Hdg^m(X)((-1)^{ m-1})$ contains the lattice on the right and the latter has no overlattices (its discriminant group has no nontrivial isotropic elements; see Section~\ref{secc1}).

\begin{rema}\upshape
Take $m=2$.\
The integral Hodge conjecture in degree $2$ for GM fourfolds  was recently proved in \cite[Corollary~1.2]{per}.\ Therefore, we get a family (parametrized by the fourfold~$Y_\A^1 $) of GM fourfolds $X$ such that all classes in $\Hdg^2(X)$ are classes of algebraic cycles.
\end{rema}

\begin{exam}\upshape
Take $m=3$ and $V_5=V_\xi$.\ We get  a GM sixfold $X^6_\A$ which can be defined 
  inside $\P(\C e_{00}\oplus \bw2V_\xi)$ by the quadratic equation
\begin{equation*}
x_{00}^2=x_{12}x_{13}+x_{23}x_{24}+x_{34}x_{35}-x_{45}x_{14}+x_{15}x_{25}
\end{equation*}
(the right side is the equation~\eqref{eqQ} of the $\G$-invariant quadric $\QQ\subset  \P(\bw2V_\xi)$) and  the Pl\"ucker quadrics in the  $(x_{ij})_{1\le i<j\le 5}$ that define  $\Gr(2,V_\xi)$ in $\P(\bw{2}V_\xi)$.\ Since the  equation of $\QQ$ is $\G$-invariant, we see that $\Z/2\Z\times \G$ acts on $\C e_{00}\oplus \bw2V_\xi$ component-wise, and this group is~$\Aut(X^6_\A)$.

The integral Hodge conjecture in degree $3$ is not known in general for GM sixfolds $X$, but it  was   proved in \cite[Corollary~8.4]{per} that the cokernel $V^3(X)$ (the {\em Voisin group}) of the cycle map
$$\CH^3(X)\lra \Hdg^3(X)
$$
is $2$-torsion.\ When $X=X^6_\A$, since the cycle map is surjective for $\Gr(2,V_\xi)$, the image of the cycle map, modulo $\gamma^* H^{2m}(\Gr(2,V_5),\Z)$,  is a $\G$-invariant, not necessarily saturated, sublattice   of $\SN$ of index a power of $2$.
\end{exam}

\section{Irrational GM threefolds}\label{sect5}
 
  \subsection{Double EPW surfaces and their automorphisms}\label{se51}
  
 Let  $Y_A\subset \P(V_6)$ be a quasi-smooth  EPW sextic, where $A\subset \bw3V_6$ is a Lagrangian subspace with no decomposable vectors.\ Its singular locus is the smooth surface $Y_A^{\ge 2}$ and, as explained in Appendix~\ref{sec3}, there is a canonical connected \'etale  double  covering 
  $\tY_A^{\ge 2}\to Y_A^{\ge 2}$.

Let $X$ be  
 any (smooth) GM variety  of dimension $3$ or $5$ associated with $A$ and let~$\Jac(X)$ be its intermediate Jacobian.\ It is a 10-dimensional \av\ endowed with a canonical principal polarization $\theta_X$.\ By \cite[Theorem~1.1]{dkij}, there is a canonical principal polarization $\theta$ on  $\Alb (\tY_A^{\ge 2})$ and a
canonical isomorphism
\begin{equation}\label{jxa}
(\Jac(X),\theta_X)\isomlra (\Alb (\tY_A^{\ge 2}),\theta)
\end{equation}
between $10$-dimensional \ppavs.\ By~\eqref{h1}, the tangent spaces at the origin of these \avs\ are isomorphic to $A$.

The subgroup $\Aut(X)$ of $\Aut(Y_A)$ (see~\eqref{autxa}) acts faithfully on both $\Jac(X) $ and $\Alb (\tY_A^{\ge 2})$ and, by~Proposition~\ref{propc8}, the isomorphism above is $\Aut(X)$-equivariant.

\subsection{Explicit irrational GM threefolds}\label{sec52}
Consider the 
\mon\  Lagrangian $\A$.\ By Proposition~\ref{prop:all_autom_A}, we have $\Aut(Y_\A)\isom  \G$ and the analytic representation  of the action of that group on $\Alb (\tY_A^{\ge 2})$   is, by~Proposition~\ref{propc7}, the  representation of $\G$ on $\A$, that is, the irreducible representation~$\bw2\xi$ of $\G$ (Section~\ref{se31}).\ In particular,~$\G$ acts faithfully on the $10$-dimensional \ppav\ 
\begin{equation}\label{defj}
(\J,\theta):=(\Alb (\tY_\A^{\ge 2}),\theta)
\end{equation}
  by automorphisms that preserve the principal polarization~$\theta$.\  By Lemma~\ref{lb3}, any $\G$-invariant polarization on $\J$ is proportional to~$\theta$.\ 

 \begin{prop}\label{prop61}
The \ppav\ $(\J,\theta)$ is indecomposable.
\end{prop}

\begin{proof}
If $(\J,\theta)$ is isomorphic to a product of $m\ge 2$ nonzero indecomposable \ppavs, such a decomposition is unique up to the order of the factors hence induces a morphism $u\colon\G\to\gS_m$ (the group $\G$ permutes the factors).\ Since  the analytic representation is irreducible, the image of $u$ is nontrivial  and, the group $\G$ being simple,  $u$ is injective; but this is impossible because~$\G$ contains elements of order $11$ but not $ \gS_m$, because  $m\le 10$.
\end{proof}

We can now prove our  main result.

 \begin{theo}\label{main}
Any smooth GM threefold associated with the Lagrangian $\A$ is irrational.
\end{theo}

\begin{proof}
Let $X$ be such a   threefold.\ By Proposition~\ref{propc7}, the isomorphism
$(\Jac(X),\theta_X)\isomto ( \J,\theta)$  in~\eqref{jxa}
 is   $\G$-equivariant.\ 
  We follow \cite{bea2,bea3}: to prove that~$X$ is not rational,  we apply the Clemens--Griffiths criterion (\cite[Corollary~3.26]{cg}); in view of Proposition~\ref{prop61}, it suffices to prove that  $( \J,\theta)$ is not the Jacobian of a smooth projective curve.\ 

Suppose $( \J,\theta)\isom (\Jac(C),\theta_C)$ for  some smooth projective curve $C$ of  genus $10$.\ The group~$\G$ then embeds  into  the group  of  automorphisms  of $(\Jac(C),\theta_C)$;  by  the  Torelli  theorem,  this group is isomorphic to $\Aut(C)$ if $C$ is hyperelliptic and to $\Aut(C)\times\Z/2\Z$ otherwise.\  Since any morphism from $\G$ to $\Z/2\Z$ is trivial, we see that $\G$ is a subgroup of $\Aut(C)$.\ This contradicts the fact that the     automorphism group of a curve of genus $10$   has order at most $432$ (\cite{lmfd}).
 \end{proof}

\begin{coro}\label{coro52}
There exists a complete family, with finite moduli morphism, parametrized by the smooth projective surface $Y^{\ge2}_{\A}$, of irrational smooth ordinary GM threefolds.
\end{coro}

\begin{proof}
This follows from the theorem and \cite[Example~6.8]{dkmoduli}.
\end{proof}

The theorem  applies in particular to the   GM threefold $X^3_\A$ defined in Section~\ref{se23}.

\begin{coro}\label{coro53}
The   GM threefold $X^3_\A$ is irrational. 
\end{coro}

\begin{rema}\label{rema53}
It is a general fact that all smooth GM varieties of the same dimension constructed from the same Lagrangian are birationally isomorphic (\cite[Corollary~4.16]{dkclass}); in particular, all threefolds in the family of Corollary~\ref{coro52} are mutually birationally isomorphic.\ 
 \end{rema}
 
 \begin{rema}\label{rema56a}
The Clemens--Griffiths component of a
 \ppav\ is the product of its indecomposable factors that are not  isomorphic to Jacobians of smooth projective curves, and the Clemens--Griffiths component of a Fano threefold is the Clemens--Griffiths component of its intermediate Jacobian; it
  follows from the Clemens--Griffiths method that the  Clemens--Griffiths component of  a Fano threefold  is a birational invariant.\ By Proposition~\ref{prop61}, the Clemens--Griffiths component of the GM threefolds constructed from the Lagrangian~$\A$ is  $(\J,\theta)$; in particular, these threefolds are not birationally isomorphic to any smooth cubic threefold (because their  Clemens--Griffiths components all have dimension $5$).
 \end{rema}

\begin{rema}\label{rema53a}
All GM fivefolds are rational (\cite[Proposition~4.2]{dkclass}).\ We do not know whether 
the smooth GM fourfolds associated with the Lagrangian $\A$ are rational (folklore conjectures say that they should be irrational, because they have no associated K3 surfaces; see Proposition~\ref{assoc}).
 \end{rema}


Let us go back to the
 10-dimensional \ppav\ $(\J,\theta)$ defined by~\eqref{defj}.\ It is acted on faithfully by the group~$\G$, and the associated analytic representation
$\G\to \GL(T_{\J,0})$
is the   irreducible 
 representation $\bw2\xi$ of $\G$ (Sections~\ref{se51} and~\ref{sec52}).\


\begin{prop}\label{prop62}
The abelian variety $\J  $ is  isogeneous to $E^{10}$, for some elliptic curve $E$.
 \end{prop} 

\begin{proof}
Since the analytic  representation is irreducible and defined over $\Q$ (Appendix~\ref{sea2}), the proposition follows from~Proposition~\ref{propb1}.
 \end{proof}

Unfortunately, we were not able to say more about the elliptic curve $E$ in the proposition: as explained in Remark~\ref{remb4}, the mere existence of a $\G$-action on $E^{10}$ with prescribed analytic representation and of a $\G$-invariant polarization does not put any restriction on  $E$.\

 We suspect that this curve $E$ is isomorphic to the elliptic curve $E_\lambda:=\C/\Z[\lambda]$, which has complex multiplication  by~$\Z[\lambda]$, where
$ \lambda:=\tfrac12(-1+\sqrt{-11})$.\
More precisely, we conjecture that~$(\J,\theta)$ is isomorphic to the \ppav\ constructed in  Proposition~\ref{prop63}.


\appendix\section{Automorphisms of double EPW sextics}\label{appC}

  \subsection{Double EPW sextics and their automorphisms}\label{se41}
  
As in Section~\ref{se1}, let $V_6$ be a $6$-dimensional complex vector space and let $A\subset \bw3V_6$ be a  Lagrangian subspace  with no decomposable vectors, with associated EPW sextic $Y_A\subset \P(V_6)$.\ There is a canonical double covering
\begin{equation}\label{piA}
\pi_A\colon \tY_A\lra Y_A
\end{equation}
  branched along the integral surface $Y^{\ge 2}_A$.\ The fourfold $\tY_A$ is called a {\em double EPW sextic} and its singular locus    is the finite set~$\pi_A^{-1}(Y^{\ge 3}_A)$ (\cite[Section~1.2]{og7} or  \cite[Theorem~B.7]{dkclass}).\ It carries the canonical polarization $H:=\pi_A^*\cO_{Y_A}(1)$ and the image of the associated morphism $\tY_A\to \P(H^0(\tY_A,H)^\vee)$  is isomorphic to $Y_A$.\   
 When   $Y_A^{\ge3}=\vide$, we   say that $Y_A$ is {\em quasi-smooth} and~$\tY_A$ is  a  smooth \hk\ variety of K3$^{[2]}$-type.

Every automorphism  of $Y_A$  induces an automorphism of~$\tY_A$ (see the proof of  \cite[Proposition~B.8(b)]{dkclass}) that fixes the   class~$H$.\ 
Conversely, let    $\Aut_H(\tY_A)$ be the group of automorphisms of $\tY_A$ that fix the class~$H$.\ It 
 contains the covering involution $\iota$ of~$\pi_A$.\ 
Any element of $\Aut_H(\tY_A)$  induces an automorphism of $\P(H^0(\tY_A,H)^\vee)\isom \P( V_6)$ hence descends to an 
 automorphism of~$Y_A$.\ This gives  a central extension  
\begin{equation}\label{central}
0\to \langle \iota\rangle \to \Aut_H(\tY_A) \to  \Aut(Y_A)\to 1.
\end{equation}
As we  will check in~\eqref{h2o}, the space $H^2(\tY_A, \cO_{\tY_A}) $ has dimension 1.\ It is acted on by the group of
automorphisms of $\tY_A$ and this defines 
another extension
\begin{equation}\label{defm}
1\to \Aut_H^s(\tY_A) \to \Aut_H(\tY_A) \to\mmu_r\to 1.
\end{equation}
The image of ~$\iota$ in $\mmu_r$ is $-1$ and $\Aut_H^s(\tY_A)$ is the subgroup of  elements of $\Aut_H(\tY_A)$ that act trivially on $H^2(\tY_A, \cO_{\tY_A}) $ (when $Y_A^{\ge3}=\vide$, these  are exactly, by Hodge theory, the symplectic   automorphisms---those that leave any symplectic $2$-form on $\tY_A$ invariant).\

We will show in the next proposition (which was kindly provided by A. Kuznetsov) that these extensions  are both trivial.\ 
For that, we construct an extension
\begin{equation}\label{exttilde}
1 \to \mmu_2 \to \widetilde\Aut(Y_A) \to \Aut(Y_A) \to 1
\end{equation}
as follows.\ Recall from~\eqref{autya} that there is an embedding $\Aut(Y_A) \hra \PGL(V_6)$.\ Let $G$ be the inverse image of $\Aut(Y_A)$ via the canonical map $\SL(V_6)\to \PGL(V_6)$.\ It is an extension of~$\Aut(Y_A)$ by~$\mmu_6$ and we set $\widetilde\Aut(Y_A):=G/\mmu_3$.\ 

The action of $G$ on $V_6$ induces
an action on $\bw3V_6$ such that $\mmu_6$ acts through its cube,
hence the latter action   factors through an action of
$
\widetilde\Aut(Y_A) $.\ 
The subspace $A \subset \bw3V_6$ is preserved by this action,
hence we have a morphism of central extensions
\begin{equation}\label{esss}
\begin{aligned}
\xymatrix
@R=5mm@M=2mm
{
1\ar[r]&\mmu_2\ar[r]\ar@{_(->}[d]&\widetilde\Aut(Y_A)\ar[r]\ar[d]&\Aut(Y_A)\ar[r]\ar[d]&1\\
1\ar[r]&\C^\times\ar[r]
&\GL(A)\ar[r]&\PGL(A)\ar[r]&1.
} 
\end{aligned}
\end{equation}
%

\begin{lemm}\label{nlem}
The vertical morphisms in~\eqref{esss} are injective.
\end{lemm}

\begin{proof}
Let $g\in  G\subset \SL(V_6)$.\ Assume that $g$ acts trivially on $A$.\ Then it also acts trivially on~$A^\vee$.\  There is a $G$-equivariant exact sequence $0 \to A \to \bw3V_6 \to A^\vee \to 0$ which splits $G$-equivariantly because~$G$ is finite.\ It follows that $G$ also acts trivially on $\bw3V_6$.\ 
The natural morphism
$\PGL(V_6) \to \PGL(\bw3V_6)$ being injective,   $g$ is in $\mmu_6$.\ Finally, $\mmu_6/\mmu_3 $
acts nontrivially on $A$, hence $g$ is in $\mmu_3$ and its image in $  \widetilde\Aut(Y_A)$ is~$1$.\ This proves that the middle vertical map in~\eqref{esss} is injective.

Assume now that $g $ acts as $\lambda \Id_A$ on $A$.\ Its eigenvalues on $\bw3V_6$ are then~$\lambda$ and $\lambda^{-1}$, both with multiplicity $10$.\ Let $\lambda_1,\dots,\lambda_6$ be its eigenvalues on $V_6$.\ For   all $1\le i<j<k\le 6$, one then has $\lambda_i\lambda_j\lambda_k= \lambda$  or $\lambda^{-1}$.\ It follows that   if $i,j,k,l,m$ are all distinct, $\lambda_i\lambda_j\lambda_k,\lambda_i\lambda_j\lambda_l,\lambda_i\lambda_j\lambda_m$ can only take 2 values, hence $\lambda_k,\lambda_l,\lambda_m$  can only take 2 values.\ So,  there are at most $2$ distinct eigenvalues and one of  the eigenspaces, say $E_{\lambda_1}$, has dimension at least~$ 3$.\ If $\lambda\ne \lambda^{-1}$, the eigenspace in $\bw3V_6$ for the eigenvalue $\lambda_1^3$, which is either $A$ or $A^\vee$, contains $\bw3E_{\lambda_1}$.\ This contradicts the fact that  $A$ and $A^\vee$ contain no decomposable vectors.\ Therefore, $\lambda= \lambda^{-1}$ and~$g$ acts as $\pm \Id_A$, and the first part of the proof implies that the image of $\pm g$  in $  \widetilde\Aut(Y_A)$ is~$ 1$.\ This proves that the rightmost vertical map in~\eqref{esss} is injective.
\end{proof}

\begin{prop}[Kuznetsov]\label{split1}
Let $A\subset \bw3V_6$ be a  Lagrangian subspace  with no decomposable vectors.\ The extensions~\eqref{central} and~\eqref{defm} are trivial and $r=2$; more precisely, there is an isomorphism
$$\Aut_H(\tY_A)\isom  \Aut(Y_A)\times \langle \iota\rangle$$
that splits~\eqref{central} and  the factor   $\Aut(Y_A)$ corresponds to the subgroup  $\Aut_H^s(\tY_A)$ of $\Aut_H(\tY_A)$.
%
\end{prop}

\begin{proof}
We briefly recall from \cite[Section~1.2]{og7} (see also \cite{dkcovers}) the construction of the double cover $\pi_A\colon \tY_A\to Y_A$.\ In the terminology of the latter article, one considers the Lagrangian subbundles $\cA_1:=A \otimes \cO_{\P(V_6)}$ and $\cA_2:=\bw2T_{\P(V_6)}(-3)$ of the trivial vector bundle $\bw3V_6 \otimes \cO_{\P(V_6)}$, and the first Lagrangian cointersection sheaf
$
\cR_1 := \coker(\cA_2\hra \cA_1^\vee)
$, a rank-$1$  sheaf with support~$Y_A$.\
One sets	 (\cite[Theorem~5.2(1)]{dkcovers})
$$
\tY_A = \Spec(\cO_{Y_A} \oplus \cR_1(-3)).
$$
In particular, one has
\begin{equation}\label{h2o}
H^2(\tY_A, \cO_{\tY_A}) \isom H^2(Y_A, \cR_1(-3)) \isom H^3(\P(V_6), \cA_2(-3))= H^3(\P(V_6), \bw2T_{\P(V_6)}(-6))\isom \C.
\end{equation}

The  subbundles $\cA_1$ and $\cA_2$ are invariant for  the action of $\widetilde\Aut(Y_A)$ on~$\bw3V_6$, hence 
the sheaf~$\cR_1 $ is $\widetilde\Aut(Y_A)$-equivariant.\ Finally,   the line bundle $\cO_{\P(V_6)}(-1) $ 
has a $G$-linearization (the subgroup $G\subset \SL(V_6)$ was defined right before Lemma~\ref{nlem}).\ It follows that $\cO_{\P(V_6)}(-3)$ has an $\widetilde\Aut(Y_A)$-linearization, hence
the same is true for the sheaf $\cR_1(-3)$.\ Therefore, the group $\widetilde\Aut(Y_A)$ acts on $\tY_A$ and fixes the polarization $H$.

Observe now that since the nontrivial element of $ \mmu_2 \subset \widetilde\Aut(Y_A)$ acts by $-1$ on $A$,  hence also on $ \cR_1$, and since it acts by $-1$ on $\cO(-1)$,  hence also
on $\cO(-3)$, it follows that $ \mmu_2$ acts trivially on $\cR_1(-3)$, hence also on $\tY_A$.\ Therefore, the morphism $\widetilde\Aut(Y_A)\to  \Aut_H(\tY_A)$ factors through the quotient
 $\widetilde\Aut(Y_A)/\mmu_2 = \Aut(Y_A)$.\ In other words, the surjection $\Aut_H(\tY_A) \to  \Aut(Y_A)$ in~\eqref{central} has a section and this central extension is trivial.

 
The action of  the group $\Aut(\tY_A)$  on the 1-dimensional vector space 
$H^2(\tY_A, \cO_{\tY_A})$ defines a   morphism $\Aut(\tY_A)\to\C^\star$ that maps  $\iota$  to $-1$.\ 
The   lift $\widetilde\Aut(Y_A)\to \Aut(Y_A) \hra \Aut_H(\tY_A)$ acts trivially on $H^2(\tY_A, \cO_{\tY_A})$ because 
  its action  is induced by the action of $\PGL(V_6)$, which has no nontrivial characters.\ This gives a surjection $\Aut_H(\tY_A) \to  \langle \iota\rangle$ which is trivial on the image of the section $\Aut(Y_A) \hra \Aut_H(\tY_A)$.\  This implies   that the extension~\eqref{defm} is also trivial and $r=2$.\ The theorem is therefore proved.
\end{proof}

\subsection{Moduli space and  period map of (double) EPW sextics} \label{sec42}
Quasi-smooth   EPW sextics admit an affine coarse moduli space ${\mathbf{M}^{\mathrm{EPW},0}}$, constructed in \cite{og5}  as a GIT quotient  by $\PGL(V_6)$ of an affine open dense subset of the space  of Lagrangian subspaces in~$\bw3V_6$.\ 

Let $\tY$ be a hyperk\"ahler fourfold of K3$^{[2]}$-type (such as a double EPW sextic).\  The lattice $H^2(\tY,\Z)$ (endowed with the Beauville--Bogomolov quadratic form $q_{BB}$) is isomorphic to the lattice
\begin{equation}\label{defL}
L:=U^{\oplus 3}\oplus E_8(-1)^{\oplus 2}\oplus (-2), 
\end{equation} 
where $U$ is the hyperbolic plane $\bigl( \Z^2,  \bigl(\begin{smallmatrix}  0& 1\\ 1 & 0 \end{smallmatrix}\bigr)\bigr)$, $ E_8(-1)$ is the negative definite even rank-8 lattice, and $(m)$ is the rank-$1$ lattice with generator of square $m$.\

Fix a class $h\in L$ with $h^2=2$.\ These classes are all in the same $O(L)$-orbit and
\begin{equation}\label{defhperp}
h^\bot \isom U^{\oplus 2}\oplus E_8(-1)^{\oplus 2}\oplus (-2)^{\oplus 2}.
\end{equation} 
The space 
$$
\begin{aligned}
\Omega_{h} :={}& \{ [x]\in \P(L \otimes \C)\mid  x\cdot h=0,\ x\cdot x=0,\ x\cdot \bar x>0\}\\
{}={}& \{ [x]\in \P(h^\bot \otimes \C)\mid  x\cdot x=0,\ x\cdot \bar x>0\}
\end{aligned}
$$
 has two connected components, interchanged by complex conjugation, which are Hermitian symmetric domains.\ 
 It is acted on by the   group 
 $$  \{g\in  O(L)\mid g(h)=h\},$$
also with  two connected components, which is also the index-2 subgroup $\widetilde O(h^\bot)$ of  $O(h^\bot)$  that consists of isometries that act trivially on the discriminant group $\Disc(h^\bot)\isom (\Z/2\Z)^2$.\
The quotient is an irreducible quasi-projective variety (Baily--Borel) and the {\em period map}
 \begin{equation}\label{defp}
 \wp\colon {\mathbf{M}^{\mathrm{EPW},0}} \lra  \widetilde O(h^\bot)\backslash \Omega_{h},\quad [\tY]\longmapsto [H^{2,0}(\tY)]
 \end{equation}
is algebraic (Griffiths).\ It is  an open embedding by
 Verbitsky's Torelli theorem   (\cite{ver, marsur, huybki}).\
 
 If $A\subset \bw3V_6$ is a Lagrangian such that $\tY_A$ is  smooth   with period $[x]\in\P(L \otimes \C)$ (well defined only up to the action of $ \widetilde O(h^\bot)$), the Picard group $\Pic(\tY_A)$ is, by Hodge theory, isomorphic to $x^\bot\cap L$.\ It contains the class $h$ (of square $2$) but, as explained in \cite[Theorem~5.1]{dm}, no class orthogonal to $h$ of square $-2$.

\subsection{Automorphisms of prime order }\label{secc1} 

Let $\tY$ be a hyperk\"ahler fourfold of K3$^{[2]}$-type.\  In the lattice $(H^2(\tY,\Z),q_{BB})$ mentioned in Appendix~\ref{sec42}, we consider
 the {\em transcendental lattice} 
$$\Tr(\tY) :=\Pic(\tY)^\bot\subset H^2(\tY,\Z) .$$ 
 
 The automorphism  group $\Aut(\tY)$ acts faithfully by isometries on the lattice $(H^2(\tY,\Z),q_{BB})$ and preserves the sbulattices $\Pic(\tY)$ and $\Tr(\tY)$.\
If $G$ is a subset of~$\Aut(\tY)$, we denote by
$T_G(\tY)$
the invariant lattice (of  elements of~$H^2(\tY,\Z) $ that are invariant by all elements of $G$) and by $S_G(\tY):=T_G(\tY)^\bot$ its orthogonal in $H^2(\tY,\Z) $.\

Many results are known about automorphisms of prime order $p$ of \hk\ fourfolds.\ We  restrict ourselves to the case $p\ge 11$.\ In the statement below,  the  rank-$20$ lattice $\SN$    was defined in \cite[Example~2.9]{mon3} by an explicit $20\times 20$ Gram matrix (see also \cite[Example~2.5.9]{monphd}); it is negative definite, even,  contains no $(-2)$-classes,  its discriminant group is $(\Z/11\Z)^2$, and its discriminant form is $\left(\begin{smallmatrix} -2/11 & 0\\ 0 & -2/11 \end{smallmatrix}\right)$.\

\begin{theo}\label{thc1}
Let $\tY$ be a projective hyperk\"ahler fourfold of K3$^{[2]}$-type and let $g$ be a symplectic   automorphism of~$\tY$ of prime order $p\ge11$.\  There are inclusions $\Tr(\tY)\subset T_g(\tY)$ and $S_g(\tY)\subset \Pic(\tY)$, and \mbox{$p= 11$.}\ The  lattice $S_g(\tY)$ is 
 isomorphic to $\SN$ and   $\rho(\tY)=21$.\ The possible lattices~$T_g(\tY)$ are
$$ \begin{pmatrix}
2 &1 &0  \\
1  &6&0\\
0&0&22
\end{pmatrix}\quad or \quad
 \begin{pmatrix}
6&2 &2  \\
2  &8&-3\\
2&-3&8
\end{pmatrix}.$$
\end{theo} 

\begin{proof} 
The proof  is a compilation of previously known results on symplectic automorphisms.\ 
The  bound  $p\le 11$ is \cite[Corollary~2.13]{mon3}.\ The inclusions and the properties of  the lattice~$S_g(\tY)$ are in      \cite[Lemma~3.5]{mon2}, the equality $\rho(\tY)=21$ is in \cite[Proposition~1.2]{mon3},   the lattice $S_g(\tY)$ is determined in \cite[Theorem~7.2.7]{monphd}, and  the possible lattices $T_g(\tY)$  in \cite[Section~5.5.2]{bns}.
\end{proof}
 
This theorem applies in particular to   (smooth) double EPW sextics~$\tY_A$.\ We are interested in automorphisms that preserve the  canonical degree-2 polarization $H$.\ By Proposition~\ref{split1}, the group of   these automorphisms, modulo the covering involution $\iota$, is isomorphic to the group of automorphisms of the EPW sextic~$Y_A$.

 \begin{coro}\label{th14}
  Let  $\tY_A $ be a smooth double EPW sextic  and let $g$ be an automorphism of $\tY_A$ of prime order $p\ge 11$ that fixes the polarization~$H$.\  Then $p=11$ and\,\footnote{In the given decomposition of the lattice $\Pic(\tY_A)$, the summand $(2)$ is {\em not} generated by the polarization $H$, because~$\SN$ contains no $(-2)$-classes.}
 $$
\begin{aligned}
 S_{g}(X) \isom \SN,\qquad T_{g}(\tY_A)&\isom  \begin{pmatrix}
2 &1   \\
1  &6
\end{pmatrix}\oplus (22),\qquad \Tr(\tY_A)\isom  (22)^{\oplus 2},
\\
\Pic(\tY_A)= \Z H \oplus \SN&\isom  (2)\oplus E_8(-1)^{\oplus 2}\oplus \begin{pmatrix}
-2 &-1   \\
-1  &   -6
\end{pmatrix}^{\oplus 2}  . 
\end{aligned}
$$
In particular,   the fourfold $\tY_A$ has maximal Picard number  $21$.\ 
\end{coro}

\begin{proof}  By Proposition~\ref{split1}, the automorphism $g$ is symplectic (all nonsymplectic automorphisms have even order).\  Since $H\in T_{g}(\tY_A)$ and $q_{BB}(H)=2$, and the second lattice in Theorem~\ref{thc1} contains no classes of square $2$, there is only one possibility for~$T_{g}(\tY_A)$   (see also \cite[Section~7.4.4]{monphd}).\ There are only two (opposite) classes of square 2 in that lattice, so we find $
\Tr(\tY_A)$ as their orthogonal.\ 

We know that $\Pic(\tY _A)$ is an overlattice of $\Z  H\oplus S_{g} (\tY_A)$.\ Since the latter has  no nontrivial overlattices (its discriminant group has no nontrivial isotropic elements), they are equal.\ Finally,  the negative definite lattices $\SN$ and 
$$S:=     E_8(-1)^{\oplus 2}\oplus \begin{pmatrix}
-2 &-1   \\
-1  &   -6
\end{pmatrix}^{\oplus 2}$$
are in the same genus.\footnote{By Nikulin's celebrated result \cite[Corollary~1.9.4]{nik}, this means that they have same ranks, same signatures, and that their discriminant forms coincide.}\ They are not isomorphic (because $\SN$ does not represent $-2$) but  the indefinite lattices $(2)\oplus \SN$ and $(2)\oplus S$ are by \cite[Corollary~1.13.3]{nik}.
\end{proof}

We prove in Theorem~\ref{th47} that the  double EPW sextic $\tY_\A$ is the only smooth double   EPW sextic with an automorphism  of   order $  11$  that fixes the  polarization $H$.\

In Hassett's terminology (recalled in \cite[Section~4]{dm}), a (smooth) double EPW sextic $\tY_A$ is {\em special of discriminant $d$} if there exists a primitive rank-2 lattice $K\subset \Pic(\tY_A)$ containing the polarization $H$ such that $\disc(K^\bot)=-d$ (the orthogonal is taken in $(H^2(\tY_A,\Z),q_{BB})$);  this may only happen when $d\equiv 0,2,4\pmod{8}$ and $d>8$ (\cite[Proposition~4.1 and Remark~6.3]{dm}).\ The fourfold $\tY_A$ has an {\em associated K3 surface} if moreover the lattice $K^\bot$ is isomorphic to the opposite of the primitive cohomology lattice
of a pseudo-polarized K3 surface (necessarily of degree $d$); a necessary condition for this to happen is $d\equiv  2,4\pmod{8}$ (this was proved in \cite[Proposition~6.6]{dim} for GM fourfolds but the  computation is the same).

\begin{prop}\label{assoc}
The double EPW sextic $\tY_\A$ is special of   discriminant $d $ if and only if $d$ is a  multiple of $8$ greater than~$8$.\ In particular, it has no associated K3 surfaces.
\end{prop}

\begin{proof}
Assume that $\tY_\A$ is special of   discriminant $d $.\ Since $\Pic(\tY_\A)\isom \Z H\oplus \SN$,   the required lattice~$K$ as above is of the form $\langle H,\kappa\rangle$, where $\kappa\in \SN$ is primitive.\ Since $\Disc(\SN)\isom (\Z/11\Z)^2$, the divisibility   $\div_{\SN}(\kappa)$  divides $11$ and, since $\Disc(H^\bot)\isom (\Z/2\Z)^2$ (see \cite[(1)]{dm}), the  divisibility   $\div_{H^\bot}(\kappa)$ divides~$2$, but also divides $\div_{\SN}(\kappa)$ (because $\SN\subset H^\bot$).\ It follows that $\div_{H^\bot}(\kappa)=1$.\ The lattice $\langle H,\kappa\rangle^\bot$ therefore has discriminant~$4\kappa^2$  by the formula \cite[(4)]{dm}.\ 

It follows that  $\tY_\A$ is special of discriminant $d$ if and only if $d\equiv 0 \pmod8$ and $\SN$ primitively represents $-d/4$.\
A direct computation shows that the lattice $\SN$ contains the rank-$5$ lattice with diagonal quadratic form 
  $(-4,-4,-4,-6,-8)$.\ By \cite[Section 6(iii)]{Bharg}, the quadratic form on the last four   variables  represents every even negative integer with the exception of $-2$, and the first variable can be used to ensure that all these integers can  be primitively represented.\ This proves the proposition.
\end{proof}

   \subsection{Double EPW surfaces and their automorphisms}\label{sec3}
   
 Let  $Y_A\subset \P(V_6)$ be an EPW sextic, where $A\subset \bw3V_6$ is a Lagrangian subspace with no decomposable vectors.\ By \cite[Theorem~5.2(2)]{dkcovers}, there is a canonical connected double  covering 
\begin{equation}\label{piA2}
\tY_A^{\ge 2}\lra Y_A^{\ge 2}
\end{equation}
 between integral surfaces,  with covering involution $\tau$, branched over the finite set $Y_A^{\ge 3}$.
 
  
  We  compare automorphisms of $Y_A$ with those of $\tY_A^{\ge 2}$.\ Any automorphism of $Y_A$ induces an automorphism of its singular locus $Y_A^{\ge 2}$.\ This defines a morphism 
  $\Aut(Y_A)\to \Aut(Y_A^{\ge 2})$.\ Since $\Aut(Y_A)$ is a subgroup of $\PGL(V_6)$ and the surface 
 $Y_A^{\ge 2}$ is not contained in a hyperplane, this morphism is injective. 

 
 \begin{prop}[Kuznetsov]\label{split2}
 Let $A\subset \bw3V_6$ be a Lagrangian subspace with no decomposable vectors.\ Any element
 of $\Aut(Y_A)$ lifts to an automorphism of  $\tY_A^{\ge 2}$.\ These lifts form a subgroup  of 
   $  \Aut(\tY_A^{\ge 2})$  which is  isomorphic to the group  $\widetilde\Aut(Y_A)$ in the extension~\eqref{exttilde} via an isomorphism that takes $\langle \tau\rangle$ to $\mmu_2$.
%
\end{prop}

 \begin{proof}
 The proof follows the exact same steps as the proof of Proposition~\ref{split1}, whose notation we keep.\ By \cite[Theorem~5.2(2)]{dkcovers}, the surface $\tY_A^{\ge 2}$ is defined as
\begin{equation}\label{y2}
\tY^{\ge 2}_A = \Spec(\cO_{Y^{\ge 2}_A} \oplus \cR_2(-3)),
\end{equation}
where 
$
\cR_2 = (\bw2\cR_1\vert_{Y^{\ge 2}_A})^{\vee\vee}$.\ As in the proof of Proposition~\ref{split1}, the group $\widetilde\Aut(Y_A) $  acts on $\tY^{\ge 2}_A$ and the  nontrivial element of  $ \mmu_2  $ acts by~$-1$ on both~$\cR_1$
and   $\cO(-3)$.\ It follows that it acts by~$1$ on $ \cR_2$ and by $-1$ on $ \cR_2(-3)$, hence as the involution~$\tau$ on $\tY^{\ge 2}_A$.\ This proves  the proposition.
\end{proof}

   It is possible to deform the double cover~\eqref{piA2} to the canonical double \'etale covering associated with the (smooth) variety of lines on a quartic double solid (see the proof of \cite[Proposition~2.5]{dkij}), so we can use Welters' calculations  in \cite[Theorem (3.57) and Proposition~(3.60)]{wel}.\ In particular,      the abelian group $H_1(\tY_A^{\ge 2},\Z)$ is free of rank $20$ (and $\tau$  acts as $-\Id$) and  there are canonical isomorphisms (\cite[Proposition~2.5]{dkij})
\begin{equation}
\begin{aligned}
T_{\Alb (\tY_A^{\ge 2}),0}&\isom H^1(\tY_A^{\ge 2},\cO_{\tY_A^{\ge 2}})\isom A,
\label{h1}\\
H^2(Y_A^{\ge 2},\C)&  \isom \bw2 H^1(\tY_A^{\ge 2},\C)\isom \bw2(A\oplus \bar A).
\end{aligned}
\end{equation}
The Albanese variety $\Alb (\tY_A^{\ge 2})$ is thus an \av\ of dimension 10 and one can consider the analytic representation (see Section~\ref{sectb1})
 $$\rho_a\colon \Aut(\tY_A^{\ge 2})\lra \GL(T_{\Alb (\tY_A^{\ge 2}),0})\isom\GL(A).
$$
Recall from Proposition~\ref{split2} that there is an injective morphism $\widetilde\Aut(Y_A)\hra \Aut(\tY_A^{\ge 2})$.

\begin{prop}\label{propc7}
Let $Y_A$ be a quasi-smooth EPW sextic.\ 
The restriction of the analytic representation $\rho_a$ to the subgroup $\widetilde\Aut(Y_A)$ of $ \Aut(\tY_A^{\ge 2})$ is the injective middle vertical map in the diagram~\eqref{esss}.
\end{prop}
 
 \begin{proof}
 The morphism $\rho_a$ is the  representation of the group $\Aut(\tY_A^{\ge 2})$ on the vector space
  $$
T_{\Alb (\tY_A^{\ge 2}),0}\isom H^1(\tY_A^{\ge 2},\cO_{\tY_A^{\ge 2}}).
$$
As in  the proof of \cite[Proposition~2.5]{dkij}), there are canonical isomorphisms
 $$
H^1(\tY_A^{\ge 2},\cO_{\tY_A^{\ge 2}})\isom   H^1(Y_A^{\ge 2},\cR_2(-3)) \isom   H^1(Y_A^{\ge 2},\cO_{Y_A^{\ge 2}}(3))^\vee 
,
$$
where the first isomorphism comes from~\eqref{y2} and the second one from Serre duality (because~$ \cR_2 $ is the canonical sheaf of $Y_A^{\ge 2} $).

As in the proof of Proposition~\ref{split2}, the   sheaf $ \cO_{Y_A^{\ge 2}}(3)$ has an  $\widetilde\Aut(Y_A) $-linearization, where~$\Aut(Y_A)$ acts on  $Y^{\ge 2}_A$ by restriction and the  nontrivial element of~$ \mmu_2  $ acts by~$-1$ on~$ \cO_{Y_A^{\ge 2}}(3)$.\  

By construction, the resolution 
$$0\to (\bw2\cA_2)(-6)\to (\cA_1^\vee\otimes \cA_2)(-6)\to (\Sym^2\!\cA_1)(-6)\oplus\cO_{\P(V_6)}(-6)\to 
\cO_{\P(V_6)}\to \cO_{Y_A^{\ge 2}}\to 0
$$
  given in \cite[(33)]{dkeven} is $\widetilde\Aut(Y_A) $-equivariant, hence induces an   $\widetilde\Aut(Y_A)$-equivariant isomorphism
$$H^1(Y_A^{\ge 2},\cO_{Y_A^{\ge 2}}(3))\isom H^3(\P(V_6),(\cA_1^\vee\otimes \cA_2)(-3)) =A^\vee\otimes H^3(\P(V_6), \cA_2(-3)).
$$
As already noted during the proof of Proposition~\ref{split1}, $\widetilde\Aut(Y_A)$ acts trivially on the $1$-dimensional vector space $H^3(\P(V_6), \cA_2(-3))=H^3(\P(V_6), \bw2T_{\P(V_6)}(-6))$.\ All this proves that the action of~$\Aut(\tY_A^{\ge 2})$ on $T_{\Alb (\tY_A^{\ge 2}),0}$ is indeed given by the desired morphism.
\end{proof}

 \subsection{Automorphisms of GM varieties}\label{appc4}
 Let as before $V_6$ be a 6-dimensional vector space and let $A\subset \bw3V_6$ be a Lagrangian subspace with no decomposable vectors.\ Let $V_5\subset V_6$ be a hyperplane and let $X$ be the associated (smooth ordinary) GM variety (Section~\ref{se22n}).\ One has (see~\eqref{autxa})
    \begin{equation*}
   \Aut(X)\isom \{ g\in  \PGL(V_6)\mid \bw3g(A)=A,\ g(V_5)=V_5\}.
    \end{equation*}
Since the extension~\eqref{exttilde}  splits (Proposition~\ref{split1}), there is a lift
        \begin{equation}\label{repx2}
   \Aut(X)\lra  \GL(A)
    \end{equation}
(see~\eqref{esss})  which  is injective by Lemma~\ref{nlem}.

 When the dimension of $X$ is either $3$ or $5$,  its intermediate Jacobian~$\Jac(X)$  is a 10-dimensional \av.\ 
 By \cite[Theorem~1.1]{dkij}, it is 
canonically isomorphic to
$
\Alb (\tY_A^{\ge 2})
$ (see~\eqref{jxa}).\ Therefore, there is
 an isomorphism
\begin{equation*}
T_{\Jac(X),0}\isomlra T_{\Alb (\tY_A^{\ge 2}),0}.
\end{equation*}
Together with the isomorphism~\eqref{h1}, this gives an  
 analytic representation 
 $$\rho_{a,X}\colon \Aut(X)\lra \GL(T_{\Jac(X),0})\isomlra \GL(A).$$
    
\begin{prop}\label{propc8}    
The analytic representation 
 $\rho_{a,X}$ coincides with the injective morphism~\eqref{repx2}.\ Equivalently, the isomorphism ~\eqref{jxa} is $\Aut(X)$-equivariant.
\end{prop}

\begin{proof}
Assume $\dim(X)=3$ and choose a  line $L_0\subset X$.\ The isomorphism  $
\Alb (\tY_A^{\ge 2})\isomto \Jac(X)
$ was then constructed in \cite[Theorem~4.4]{dkij} from the Abel--Jacobi map
$$\AJ_{Z_{L_0}}\colon H_1( \tY_A^{\ge 2},\Z)\lra H_3(X,\Z)
$$
 associated with   a family $Z_{L_0}\subset X\times \tY_A^{\ge 2}$  of curves on $X$ parametrized by $\tY_A^{\ge 2}$.\ Although the family $Z_{L_0}$ does depend on the choice of $L_0$, the   map $\AJ_{Z_{L_0}}$  does not.

Let $g\in \Aut(X)$ (also considered as an automorphism  of $\tY_A^{\ge 2}$).\ By the functoriality properties of the Abel--Jacobi map (\cite[Lemma~3.1]{dkij}), we obtain
$$\AJ_{Z_{L_0}}\circ g_*=\AJ_{(\Id_{X}\times g)^*(Z_{L_0})}=
\AJ_{( g\times \Id_{\tY_A^{\ge 2}})_*(Z_{g^{-1}(L_0)})}=g_*\circ \AJ_{Z_{g^{-1}(L_0)}},
$$
which proves the proposition.\ 
 When $\dim(X)=5$, the proof is similar, except that $Z_{\Pi_0}$ is now a family of surfaces in $X$ that depends  on a plane $\Pi_0\subset X$.
\end{proof}

 \section{Representations of the group $\G$}\label{sea2}
 
  The group $\G:=\PSL(2,\F_{11})$   is the only simple group of order $660=2^2\cdot 3\cdot 5\cdot 11$.\ 
   It can be  generated by the classes  
 $$ a=\begin{pmatrix}
5 &0   \\
0  &  9  
\end{pmatrix},\quad 
b=\begin{pmatrix}
3 &5   \\
-5 & 3
\end{pmatrix},\quad
c=\begin{pmatrix}
 1 & 1   \\
0  &   1   
\end{pmatrix},
$$
and $a^5=-b^6=c^{11}=I_2$, the identity matrix.

  The group $\G$ has 8 irreducible $\C$-representations, of dimensions  $1$, $5$, $5$, $10$, $10$, $11$, $12$, and~$12$.\  Here is a  character table for four of these
  irreducible representations.
    \begin{table}[h]
\renewcommand\arraystretch{1.5}
 \begin{tabular}{|c|c|c|c|c|c|c|c|c|}
 \hline
Conjugation class&$[I_2]$&$[c]$&$[c^2]$&$[a]=[a^4]$&$[a^2]=[a^3]$&$[b]=[b^5]$&$[b^2]=[b^4]$&$[b^3]$
 \\
Cardinality&$1$&$60$&$60$&$132$&$132$&$110$&$110$&$55$
\\
Order&$1$&$11$&$11$&$5$&$5$&$6$&$3$&$2$
\\

 \hline
$\chi_0$
 &$1$&$1 $&$1 $ &$1$ &$1 $ &$1  $&$1  $&$1  $
\\
 \hline
 $\xi$&$5$&$\lambda$&$\bar\lambda$&$0$&$0$&$ 1$&$-1$&$1$
\\
 \hline
 $\xi^\vee$&$5$&$\bar\lambda$&$\lambda$&$0$ &$0$ &$1$&$-1$&$1$ 
\\
 \hline
$\bw2\xi $&$ 10$&$ -1$&$ -1$&$ 0$&$ 0$&$1 $&$ 1$&$ -2$
 \\
 \hline
 \end{tabular}
\vspace{5mm}
\captionsetup{justification=centering}
\caption{Partial character table for $\G$}\label{tab1}
\end{table}
\vspace{-5mm}

As before, we have set (where $\zeta_{11}=e^{\frac{2i\pi}{11}}$)
\begin{equation*}\label{defgamma}
 \lambda:=\zeta_{11}^{1^2}+\zeta_{11}^{2^2}+\zeta_{11}^{3^2}+\zeta_{11}^{4^2}+\zeta_{11}^{5^2}=\zeta_{11}+\zeta_{11}^3+\zeta_{11}^4+\zeta_{11}^5+\zeta_{11}^9=\tfrac12(-1+\sqrt{-11}).
\end{equation*}

The representation $\xi$ has a realization in the matrix ring $\cM_5(\C)$ for which
\begin{equation}\label{real}
\xi(a)= \begin{pmatrix}
 0&0&0&0&1\\
1&0&0&0&0\\
0&1&0&0&0\\
0&0&1&0&0\\
0&0&0&1&0
\end{pmatrix},\quad
\xi(c)= \begin{pmatrix}
 \zeta_{11}&0&0&0&0\\
0&\zeta_{11}^4&0&0&0\\
0&0&\zeta_{11}^5&0&0\\
0&0&0&\zeta_{11}^9&0\\
0&0&0&0&\zeta_{11}^3
\end{pmatrix} 
.
\end{equation}

Every irreducible character  of $\G$ has Schur index 1 (\cite[\S~12.2]{ser}, \cite[Theorem~6.1]{fei}).\ 
 In particular, the 
  representation $\bw2\xi$, having an integral character, 
  can be defined over $\Q$ and even, by a theorem of Burnside (\cite{bur}), over $\Z$, that is, by a  morphism $\G\to \GL(10,\Z)$.\ The representation $\bw2\xi$ is self-dual, so there is a  $\G$-equivariant isomorphism
\begin{equation}\label{defu}
w\colon \bw2V_\xi\isomlra \bw2V_\xi^\vee,
\end{equation}
unique up to multiplication by a nonzero scalar, and it  is    symmetric (\cite[prop.~38]{ser}).\

\section{Decomposition of abelian varieties with automorphisms}\label{b2}

We gather here a few very standard notation and facts about abelian varieties.\ Let $X$ be a complex abelian variety.\ 
We denote by $\Pic(X)$  the group of isomorphism classes of line bundles on~$X$, by $\Pic^0(X)\subset \Pic(X)$  the subgroup of classes of line bundles   that are algebraically equivalent to $0$, and by $\NS(X)$  the N\'eron--Severi group  $\Pic(X)/\Pic^0(X)$, a free abelian group of finite rank.\ The group $\Pic^0(X)$ has a canonical structure of an abelian variety; it is called the dual abelian variety.\ 
Any endomorphism $u$ of $X$ induces an  endomorphism $\widehat u$ of $\Pic^0(X)$.

Given the class $\theta\in \NS(X)$  of a line bundle $L$ on $X$, we let $\phi_{\theta}$ be the  morphism 
 $$
\begin{aligned}
 X&\lra \Pic^0(X) \\ 
x& \longmapsto \tau_x^*L\otimes L^{-1}  
\end{aligned}
$$
of \avs, where $\tau_x$ is the translation  by $x$ (it is independent of the choice of the representative $L$ of $\theta$).\ When $\theta$ is a polarization, that is, when $L$ is ample,  $\phi_{\theta}$ is an  isogeny.\ 

We say that $\theta$ is a   principal polarization when $\phi_\theta$ is an isomorphism.\ If $n:=\dim(X)$, this is equivalent to saying that the self-intersection number $\theta^n$ is $n!$.\ The associated {\em Rosati involution}  on $\End(X)$ is then defined  by   
$u\mapsto u':=\phi_{\theta}^{-1}\circ \widehat u  \circ \phi_{\theta}$.\ The map
 $$
\begin{aligned}
\iota_{\theta}\colon \NS(X) &\lhra \End(X) \\ 
\theta'& \longmapsto \phi_{\theta}^{-1}\circ \phi_{\theta'}
\end{aligned}
$$
is an injective morphism of free abelian groups
whose image is the group 
 $\End^s(X)$
 of symmetric elements  for the Rosati involution (\cite[Theorem~5.2.4]{bl}).\ If $u\in \End(X)$, one has $\phi_{u^*\theta'}=\widehat u\circ \phi_{\theta'}\circ u$ hence
 \begin{equation}\label{for}
\iota_{\theta}(u^*\theta')=\phi_{\theta}^{-1}\circ \phi_{u^*\theta'}
=\phi_{\theta}^{-1}\circ \widehat u\circ \phi_{\theta'}\circ u=u'\circ \phi_{\theta}^{-1} \circ \phi_{\theta'}\circ u=
u'\circ \iota_{\theta}(\theta')\circ u.
\end{equation}

 Set $\NS_\Q(X)=\NS(X)\otimes \Q$ and $\End_\Q(X)=\End(X)\otimes \Q$ (both are finite-dimensional $\Q$-vector spaces).\ If the polarization $\theta$ is no longer principal, or if $\theta\in \NS_\Q(X)$ is only a $\Q$-polarization, the Rosati involution is still defined on $\End_\Q(X)$   by the same formula and
 we may   view  $\iota_{\theta}$ as an injective morphism
 $$
\begin{aligned}
\iota_{\theta}\colon \NS(X)_\Q &\lhra \End_\Q(X) 
 \end{aligned}
$$
 with image $\End^s_\Q(X)$ (\cite[Remark~5.2.5]{bl}).\ Formula \eqref{for} remains valid for $u\in \End(X)$ and $\theta'\in \NS(X)_\Q$.
 
 We will also need the so-called {\em analytic}    representation
\begin{equation*}
  \rho_a\colon   \End_\Q(X) \lhra\End_\C(T_{X,0}).
\end{equation*}
It sends an endomorphism of $X$ to its tangent map at $0$.

\subsection{$\Q$-actions on abelian varieties}\label{sectb1}

Let $X$ be an \av\ and let $G$ be a finite group.\  A $\Q$-action of $G$ on $X$ is a  morphism $\rho\colon \Q[G]\to \End_\Q(X)$ of $\Q$-algebras.\ The composition
$$
  G\xrightarrow{\ \rho\ }  \End_\Q(X) \xrightarrow{\ \rho_a\ } \End_\C(T_{X,0})
$$
is called the  analytic representation
 of $G$.

 \begin{prop}\label{propb1}
Let  $X$ be an \av\ of dimension~$n$ with a $\Q$-action of a finite group~$G$.\ Assume that  the analytic representation  of $G$  is irreducible and defined over~$\Q$.\ Then~$X$ is isogeneous to the product of $n$ copies of   an elliptic curve.
\end{prop}

\begin{proof}
This follows from~\cite[(3.1)--(3.4)]{es} (see also \cite[Section~1]{kr} and \cite[Proposition~13.6.2]{bl}).\ This reference  assumes that we have a bona fide action of $G$ on $X$ but only uses the induced morphism $\Q[G]\to \End_\Q(X)$ of $\Q$-algebras.
 \end{proof}

In the situation of Proposition~\ref{propb1}, we prove that any    $G$-invariant $\Q$-polarization is essentially unique.

 \begin{lemm}\label{lb3}
 Let $X$ be an \av\ with a  $\Q$-action  of  a finite group  $G$  and let~$\theta $ be a $G$-invariant polarization on $X$.\ If the analytic representation  of $G$ is irreducible, any $G$-invariant $\Q$-polarization on $X$ is a rational multiple of $\theta$.
\end{lemm}

\begin{proof}
 Let $g\in G$, which we view as an invertible element of $\End_\Q(X) $.\ Since~$\theta$ is $g$-invariant, identity~\eqref{for} (applied with $\theta'=\theta$ and $u=g$) implies $g'\circ g=\Id_X$.\ Let 
 $\theta'\in \NS(X)_\Q$.\ Applying~\eqref{for} again, we get  
 $$\iota_{\theta}(g^*\theta')=g'\circ \iota_{\theta}(\theta')\circ g=g^{-1}\circ \iota_{\theta}(\theta')\circ g.$$
 If $\theta'$ is $G$-invariant, we obtain $\iota_{\theta}(\theta')=g^{-1}\circ \iota_{\theta}(\theta')\circ g$ for all $g\in G$.\ 
  If the analytic representation   of $G$
   is irreducible, $\rho_a(\iota_{\theta}(\theta'))$ must, by   Schur's lemma, be a multiple of the identity, hence  $\theta'$ must be a  multiple of $\theta$.
\end{proof}

\subsection{Polarizations on self-products of elliptic curves}
Let $E$ be an elliptic curve, so that $\go_E:=\End(E)$ is either $\Z$ or an order in an imaginary quadratic extension of $\Q$.\ We have
$$\End(E^n)\isom \cM_{n}(\go_E)\quad\textnormal{and}\quad \End_\Q(E^n)\isom \cM_{n}(\go_E\otimes\Q), $$
and $\rho_a$ is the  embedding of these matrix rings into the  ring $\cM_n(\C)$ induced by the choice of an embedding $\go_E\hra \C$.

Polarizations on $E^n$ were studied in particular by Lange in~\cite{lan}.\ We denote by $\theta_0$  the product principal polarization on $E^n$.

\begin{prop}\label{propb3}
Let $E$ be an elliptic curve.\ 
\begin{itemize}
\item The Rosati involution defined by $\theta_0$ on $\End(E^n)$   corresponds to the involution $M\mapsto \overline M^T$ on $\cM_{n}(\go_E)$.
\item Via the embedding $\iota_{\theta_0}$, polarizations $\theta$ on $E^n$ correspond to positive definite Hermitian matrices   $M_\theta\in\cM_{n}(\go_E)$ and the degree of the polarization $\theta$ is $ \det(M_\theta)$.
\item The group of automorphisms   $\Aut(E^n,\theta)$ is the unitary group
$$\U(n,M_\theta):= \{M\in \cM_{n}(\go_E)\mid \overline M^T M_\theta\, M=  M_\theta\}.$$
 \end{itemize}
\end{prop}

\begin{proof}
If we write $E=\C/(\Z\oplus \tau\Z)$, the period matrix for $E^n$ is $\begin{pmatrix}I_n&\tau I_n\end{pmatrix}$.\ The first item then follows from \cite[Lemma~2.3]{lan} and elements of  $ \NS(E^n)$   correspond to Hermitian matrices.\ By \cite[Theorem~5.2.4]{bl}, polarizations correspond to positive definite Hermitian matrices and the degree of the polarization is the determinant of the matrix.\ More precisely, one has (\cite[Proposition~5.2.3]{bl})
$$\det(T I_n-M_\theta)=\sum_{j=0}^n (-1)^{n-j}\frac{\theta_0^j\cdot \theta^{n-j}}{j!(n-j)!}\,T^j
.$$
The last item follows from~\eqref{for}.
\end{proof}

\begin{rema}\label{remb4}
Let $G$ be a finite group with a $\Q$-representation $\rho\colon \Q[G]\to \cM_{n}(\Q)$.\ For any elliptic curve $E$, this defines a $\Q$-action of $G$ on $E^n$.\ 
It follows from the proposition that 
any  positive definite symmetric matrix   $M_\theta\in\cM_{n}(\Q)$ such  that, for all $g\in G$, 
$$\rho(g)^T M_\theta\, \rho(g)=  M_\theta$$
defines a   $G$-invariant $\Q$-polarization on $E^n$.\ 
Such a matrix always exists: take for example $M_\theta:=\sum_{g\in G} \rho(g)^T   \rho(g)$ (it corresponds to the~$\Q$-polarization $ \sum_{g\in G} g^*\theta_0$).\

The analytic representation is $\rho_\C\colon \C[G]\to \cM_{n}(\C)$.\ If it is irreducible, every    $G$-invariant $\Q$-polarization on $E^{  n}$ is, by Lemma~\ref{lb3}, a rational multiple of $\theta$.\ 
\end{rema}

 We end this section with the construction of an explicit \av\ of dimension~$10$ with a $\G$-action, such that the associated analytic representation is the  irreducible representation $\bw2\xi$, together with a $\G$-invariant {\em principal} polarization.\
Set $ \lambda:=\tfrac12(-1+\sqrt{-11})$ and consider the elliptic curve $E_\lambda:=\C/\Z[\lambda]$, which has complex multiplication  by~$\Z[\lambda]$.\

 \begin{prop}\label{prop63}
There exists a   principal polarization  $\theta$ on the \av\       $E_\lambda^{10}$ and a faithful action $\G\hra\Aut(E_\lambda^{10},\theta )$ such that the associated analytic representation is the  irreducible representation $\bw2\xi$ of $\G$.
\end{prop}

\begin{proof}
 By \cite[Table~1]{sch}), there is a   positive definite unimodular $\Z[\lambda]$-sesquilinear Hermitian form~$H'$ on~$\Z[\lambda]^5$ with an automorphism of order $11$.\  Its Gram matrix in the canonical $\Z[\lambda]$-basis  $(e_1,\dots,e_5)$  of~$\Z[\lambda]^5$ is 
$$
\begin{pmatrix}
 3 & 1-\bar\lambda  &-\lambda&1&-\bar\lambda   \\
1-\lambda &  3 & -1& - \lambda&1  \\
-\bar\lambda  & -1 &  3&\lambda&-1+\lambda  \\
 1 &  -\bar\lambda&  \bar\lambda&3&1-\bar \lambda \\
- \lambda & 1 &  -1+\bar\lambda& 1-\lambda&3
\end{pmatrix}
 $$
 and its unitary group has order $2^{3}\cdot 3\cdot 5\cdot 11=1\,320 $ (\cite{sch}).

By Proposition~\ref{propb3}, this form defines a principal polarization $\theta'$ on the  \av~$E_\lambda^5$ and the group $\Aut(E_\lambda^5,\theta')$ has order $1\,320 $; in particular, it contains an element of order~$11$.\ It follows from \cite{bb} that the group $\Aut(E_\lambda^5,\theta')$ is isomorphic to $\G\times\{\pm 1\} $ and  the
faithful representation   $\G\hra\Aut(E_\lambda^5,\theta')\hra \U(5,H')$ given by Proposition~\ref{propb3} is  $\xi$.\footnote{The principally polarized abelian fivefold $(E_\lambda^5,\theta')$ was studied in \cite{adl,adls,gon,rou}: it is the intermediate Jacobian of the  Klein  cubic threefold  
with equation
 $x_1^2x_2+x_2^2x_3+x_3^2x_4+x_4^2x_5 +x_5^2x_1 =0$ in $\P^4$.}

The Hermitian form $H'$ on $\Z[\lambda]^5$ induces a positive definite unimodular Hermitian form~$H$ on $\bw2\Z[\lambda]^5=\Z[\lambda]^{10}$ by the formula
$$
H(x_1\wedge x_2,x_3\wedge x_4):=H'(x_1,x_3)H'(x_2,x_4)-H'(x_1,x_4)H'(x_2,x_3).
$$
The matrix of $H$ (in the basis $(e_{12},e_{13},e_{14},e_{15},e_{23},e_{24},e_{25},e_{34},e_{35},e_{45})$) is
\begin{equation}\label{mat10}
\left(\begin{smallmatrix}
4& 2\lambda&-1 -2\lambda&-1-\lambda&-2+2\lambda &-\lambda &-1-2 \lambda&-2-\lambda &1& -2 \\
2\bar\lambda&6 &-1+2\lambda &-1+2\lambda&6+2\lambda &-2+\lambda &-4+\lambda&\lambda &-\lambda &2+\lambda  \\
-1-2\bar\lambda&-1+2\bar\lambda & 8&5+2\lambda&-2-2\lambda & 5+2\lambda &3+2\lambda&1-2\lambda &1&-1-2\lambda   \\
-1-\bar\lambda&-1+2\bar\lambda & 5+2\bar\lambda& 6&-1-2 \lambda&4&5+2\lambda&-1-\lambda &-1-\lambda&-1-\lambda \\
-2+2\bar\lambda& 6+2\bar\lambda&-2-2\bar\lambda &-1-2\bar \lambda & 8&2\lambda & -2+3\lambda& 2\lambda& -2-\lambda&3+\lambda  \\
-\bar\lambda & -2+\bar\lambda& 5+2\bar\lambda &4 & 2\bar\lambda& 6& 5+2\lambda&0 &-1 &-\lambda   \\
-1-2 \bar\lambda&-4+\bar\lambda &3+2\bar\lambda & 5+2\bar\lambda&-2+3\bar\lambda &5+2\bar\lambda & 8&2 &-1+\lambda&-1 -2\lambda  \\
-2-\bar\lambda &\bar\lambda &1-2\bar\lambda &-1-\bar\lambda & 2\bar\lambda &0 &2 & 6&2+2 \lambda & -2\lambda  \\
1&-\bar\lambda &1 &-1-\bar\lambda & -2-\bar\lambda & -1&-1+\bar\lambda &2+2\bar \lambda  &4 & -2  \\
-2&2+\bar\lambda&-1-2\bar\lambda&-1-\bar\lambda&3+\bar\lambda  &-\bar\lambda  &-1 -2\bar\lambda & -2\bar\lambda &-2 &  4 
\end{smallmatrix}\right).
\end{equation}
 By Proposition~\ref{propb3} again, the form $H$ defines a principal polarization $\theta$ on the  \av~$E_\lambda^{10}$, the group $\Aut(E_\lambda^{10},\theta)$ contains $\G $, and the corresponding analytic representation is~$\bw2\xi$.\  
\end{proof}
 
The $\G$-action on~$E_\lambda^{10}$ in the proposition is not the $\G$-action described in Remark~\ref{remb4} (otherwise, since $\G$-invariant polarizations are proportional, the matrix~\eqref{mat10} would, by Lemma~\ref{lb3}, have rational coefficients): these actions are only conjugate by a $\Q$-automorphism  of $E_\lambda^{10}$.

\end{document}